\documentclass[11pt]{amsart}
\usepackage[margin=1.5in]{geometry}

\usepackage{amsmath}

\usepackage[textsize=scriptsize]{todonotes}
\usepackage{amssymb}
\usepackage{amsthm}
\usepackage{amsfonts}
\usepackage{mathtools}
\usepackage{mathrsfs}
\usepackage{bbm}
\usepackage{comment}
\usepackage{cite}
\usepackage{hyperref}
\usepackage{cleveref}
\usepackage{tikz}
\usepackage{tikz-cd}
\usepackage{graphicx}

\DeclarePairedDelimiter\abs{\lvert}{\rvert}%
\makeatletter
\let\oldabs\abs
\def\abs{\@ifstar{\oldabs}{\oldabs*}}

\newcommand{\vertiii}[1]{{\left\vert\kern-0.25ex\left\vert\kern-0.25ex\left\vert #1 
    \right\vert\kern-0.25ex\right\vert\kern-0.25ex\right\vert}}

\def\ZZ{\mathbb{Z}}

\def\QQ{\mathbb{Q}}
\def\Q{\mathcal{Q}} 
\def\Qh{\widehat{Q}}

\def\11{\mathds{1}}

\def\FF{\mathbb{F}}

\def\CC{\mathcal{C}}

\def\DD{\mathcal{D}}

\def\CP{\mathbb{CP}}

\def\F*{\mathrm{Fin}_*}

\def\Tor{\mathrm{Tor}}

\def\Alg{\mathrm{Alg}}

\def\Hom{\mathrm{Hom}}

\def\epsilon{\varepsilon}

\let\phi\varphi

\def\1{\langle 1 \rangle}
\def\2{\langle 2 \rangle}

\def\xibar{\overline{\xi}}

\def\BPn{\mathrm{BP} \langle n \rangle}

\def\<{\langle}
\def\>{\rangle}
\def\sl{\mathrm{sl}}
\def\sl1{\mathrm{sl}_1}
\def\gl1{\mathrm{gl}}
\def\gl1{\mathrm{gl}_1}

\def\coh{\mathrm{coh}}
\def\HFp{\mathrm{H}\mathbb{F}_p}
\def\H{\mathrm{H}}
\def\BP{\mathrm{BP}}
\def\MU{\mathrm{MU}}
\def\bCPn{\left[\mathbb{CP}^n\right]}
\def\px{\frac{\partial}{\partial x}}
\def\py{\frac{\partial}{\partial y}}

\DeclareMathOperator*{\hash}{\#}

\newtheorem{lemma}{Lemma}[section]
\newtheorem{fact}[lemma]{Fact}

\newtheorem{thm}[lemma]{Theorem}
\newtheorem{cor}[lemma]{Corollary}
\newtheorem{prop}[lemma]{Proposition}

\newtheorem{question}[lemma]{Question}

\theoremstyle{definition}
\newtheorem{rmk}[lemma]{Remark}

\newtheorem{defin}[lemma]{Definition}

\newtheorem{nota}[lemma]{Notation}

\title{The Brown-Peterson Spectrum is not $\mathbb{E}_{2(p^2+2)}$ at odd primes}

\author{Andrew Senger}
\address{Department of Mathematics, Harvard University, Cambridge, MA, USA}
\email{senger@math.harvard.edu}

\begin{document}
\begin{abstract}
  We show that the odd-primary Brown-Peterson spectrum $\mathrm{BP}$ does not admit the structure of an $\mathbb{E}_{2(p^2+2)}$ ring spectrum and that there can be no map $\mathrm{MU} \to \mathrm{BP}$ of $\mathbb{E}_{2p+3}$ ring spectra. We also prove the same results for truncated Brown-Peterson spectra $\mathrm{BP} \langle n \rangle$ of height $n \geq 4$.
This extends results of Lawson at the prime $2$.
\end{abstract}

\maketitle

\setcounter{tocdepth}{1}
\tableofcontents

\section{Introduction}

Two important themes in modern homotopy theory are the study of structured ring spectra, such as $\mathbb{E}_\infty$ ring spectra, and chromatic homotopy theory, in which the $p$-primary Brown-Peterson spectrum $\mathrm{BP}$ plays a key role. In \cite{May}, May asked the following question:

\begin{question}
Does the Brown-Peterson spectrum admit the structure of an $\mathbb{E}_\infty$ ring spectrum?
\end{question}

This question has been seminal in the development of the theory of structured ring spectra. In an unpublished preprint \cite{Kriz}, Kriz developed the theory of topological Andr\'e-Quillen cohomology in an attempt to prove that $\mathrm{BP}$ does admit the structure of an $\mathbb{E}_\infty$ ring spectrum. While his attempt to apply his theory to $\mathrm{BP}$ did not ultimately succeed, a careful study of what exactly went wrong became the seed of a new attempt by Lawson to answer May's question in the negative; recently, this project reached maturity in Lawson's proof \cite{BPtwo} that $\mathrm{BP}$ does not admit an $\mathbb{E}_\infty$ multiplication at the prime $p=2$.

In this paper, we prove in \Cref{MainThm} that $\mathrm{BP}$ does not admit an $\mathbb{E}_\infty$ multiplication at odd primes. Our technique is akin to Lawson's and relies on the computation of a certain secondary power operation in the dual Steenrod algebra. The key input to this computation is the calculation of a certain $\mathrm{MU}$-power operation in $\mathrm{MU}_*$.

For further motivation and background, we refer the reader to the introduction of \cite{BPtwo}.

\subsection{Statement of results}
We prove two main results: one limiting the coherence of multiplicative structures on the Brown-Peterson spectrum and related spectra at odd primes, and another giving a stronger limitation on the coherence of complex orientations of such spectra.
The $2$-primary analogues of our result were proven by Lawson \cite[Theorem 1.1.2 and Remark 4.4.7]{BPtwo}.


\begin{thm}\label{MainThm}

	Let $p$ denote an odd prime. Neither the Brown-Peterson spectrum $\mathrm{BP}$, nor the truncated Brown-Peterson spectra $\BPn$ for $n \geq 4$, nor any of their $p$-adic completions admit the structure of an $\mathbb{E}_{2(p^2+2)}$ ring spectrum.

\end{thm}

We will prove \Cref{MainThm} at the end of Section \hyperref[secondary]{\ref*{secondary}}.


\begin{thm}\label{MUThm}

	Let $p$ denote an odd prime. Neither the Brown-Peterson spectrum $\mathrm{BP}$, nor the truncated Brown-Peterson spectra $\BPn$ for $n \geq 3$, nor any of their $p$-adic completions admit an $\mathbb{E}_{2p+3}$-map from the complex cobordism spectrum $\mathrm{MU}$.

\end{thm}

We will prove Theorem \hyperref[MUThm]{\ref*{MUThm}} at the end of Section \hyperref[primary]{\ref*{primary}}. Again, the $p=2$ case of this theorem is due to Lawson \cite[Remark 4.4.7]{BPtwo}.

Finally, we will prove in \Cref{cor:JNcor} the following result about the incompatibility of $p$-typical complex orientations with $\H_\infty$-structures:

\begin{cor}\label{cor:JNcorintro}
    Suppose $f : \MU_{(p)} \to E$ is a map of $\mathbb{H}_{\infty}$-ring spectra satisfying:
    \begin{enumerate}
        \item $f$ factors through the Quillen idempotent $\MU_{(p)} \to \BP$.
        \item $f$ induces a Landweber exact $\MU_*$-module structure on $E_*$.
    \end{enumerate}
    Then the induced formal group on $E_*$ has height at most $2$, i.e. $v_2$ is invertible in $E_* / (p,v_1)$.
\end{cor}

This corollary is a variation on a result of Johnson--Noel \cite[Theorem 1.3]{JN}, who prove the stronger result that $E_*$ must be a $\QQ$-algebra under the restriction $p \leq 13$. We suspect that $E_*$ must in fact be a $\QQ$-algebra for all primes $p$: see \Cref{qst:JN}.

\subsection{Outline of the paper}
In Section \hyperref[POp]{\ref*{POp}}, we carry out the computations of $\mathrm{MU}$-power operations that we will need. The main result \Cref{POp} is Theorem \hyperref[POpThm]{\ref*{POpThm}}. In Section \hyperref[DLMUStSec]{\ref*{DLMUStSec}}, we generalize results of \cite{BPtwo} to convert the $\mathrm{MU}$ power operations of Theorem \hyperref[POp]{\ref*{POp}} into Dyer-Lashof operations in $\pi_* (\HFp \wedge_{\mathrm{MU}} \HFp)$, thus obtaining Theorem \hyperref[DLMUSt]{\ref*{DLMUSt}}. At the end of this section, we apply these results to obtain Theorem \hyperref[MUThm]{\ref*{MUThm}}.

In Section \hyperref[DLRev]{\ref*{DLRev}}, we state some relations satisfied by the action of the Dyer-Lashof operations on $\mathrm{H}_* (\mathrm{MU};\FF_p)$ and $\mathrm{H}_* (\HFp; \FF_p)$. In Section \hyperref[SecRel]{\ref*{SecRel}}, we write down the relation defining the secondary operation of interest and show that it is defined on $\xibar_1 \in \mathrm{H}_* (\HFp; \FF_p)$. Finally, in Section \hyperref[compsecond]{\ref*{compsecond}}, we compute this secondary operation on $\xibar_1$ to be a nonzero multiple of $\tau_4$ modulo the $\xi_i$ by applying juggling formulae and a Peterson-Stein relation to reduce to Theorem \hyperref[DLMUSt]{\ref*{DLMUSt}}. We then deduce Theorem \hyperref[MainThm]{\ref*{MainThm}}.

\subsection{Questions}
Our work raises several interesting questions. While Theorems \hyperref[MainThm]{\ref*{MainThm}} and \hyperref[MUThm]{\ref*{MUThm}} provide upper bounds on the coherence of multiplicative structures on $\mathrm{BP}$ that are functions of $p$, the best known lower bounds \cite{E4}, which state that $\mathrm{BP}$ is an $\mathbb{E}_4$-algebra and admits an $\mathbb{E}_4$ orientation $\mathrm{MU} \to \mathrm{BP}$, do not depend on the prime $p$. So one is led to ask whether these coherence bounds are independent of $p$.

\begin{question}
	Let $\coh_{\mathrm{BP}} (p)$ denote the largest integer $n$ such that the $p$-primary $\mathrm{BP}$ admits the structure of an $\mathbb{E}_n$ ring spectrum. Is $\coh_{\mathrm{BP}} (p)$ constant in $p$? If not, how does it vary with $p$?
\end{question}

In another direction, we may ask about $\mathbb{E}_\infty$ structures on the truncated Brown-Peterson spectra $\mathrm{BP}\<n\>$. While Theorem \hyperref[MainThm]{\ref*{MainThm}} rules out the possibility of such structures for $n \geq 4$, the only known positive results state that $\mathrm{BP}\<1\>$ always admits an $\mathbb{E}_\infty$ structure (since it is the Adams summand) and that $\mathrm{BP}\<2\>$ admits an $\mathbb{E}_\infty$ structure at the primes $2$ and $3$ \cite{HL,LN}. What about the remaining cases?

\begin{question}\label{BP2Q}
	At which of the primes $p \geq 5$ does the height $2$ truncated Brown-Peterson spectrum $\mathrm{BP}\<2\>$ admit an $\mathbb{E}_\infty$ multiplication?
\end{question}

\begin{question}\label{BP3Q}
	At which primes does the height $3$ truncated Brown-Peterson spectrum $\mathrm{BP}\<3\>$ admit an $\mathbb{E}_\infty$ multiplication?
\end{question}


\begin{rmk}
	The above questions are not quite well-defined: there are many generalized truncated Brown-Peterson spectra $\mathrm{BP}\<n\>$ which are not a priori equivalent. However, Angeltveit and Lind \cite{addBPn} have shown that all choices of $\mathrm{BP}\<n\>$ are equivalent as spectra after $p$-completion, so that \Cref{BP2Q} and \Cref{BP3Q} are well-defined after $p$-completion.
\end{rmk}

Finally, we believe that the following strengthening of \Cref{cor:JNcorintro} holds:

\begin{question} \label{qst:JN}
  Can \Cref{cor:JNcorintro} be strengthened so that the conclusion is that $E_*$ is a $\QQ$-algebra? Equivalently, can the small primes condition be removed from \cite[Theorem 1.3]{JN}?
\end{question}

\subsection{Conventions}

We work throughout at a fixed odd prime $p$. We will let $\mathrm{H}$ denote the mod $p$ Eilenberg-MacLane spectrum $\HFp$ and let $\mathrm{H}_* (X)$ denote mod $p$ homology.

We let $F$ denote the universal formal group law, defined over $\MU_*$.

In terms of foundations, we follow \cite[Section 1.6]{BPtwo}, with the following caveat.
In the proof of \Cref{En}, we will allow ourselves to work freely with the language of $\infty$-categories and the notion of $\mathbb{E}_n$-ring native to this setting, as developed by Lurie \cite{HTT,HA}. To translate between Lawson's framework and ours, we pass to the underlying $\infty$-category of the model categories considered by Lawson. The compatibility of this procedure with multiplicative structures is justified by \cite[Theorem 7.10]{rect}.



\subsection{Generators of the homology and homotopy of $\mathrm{MU}$}

For the convenience of the reader, we review the relations between various sets of elements of $\pi_* (\mathrm{MU})$, $\mathrm{H}_* (\mathrm{MU}; \ZZ)$ and $\pi_* (\mathrm{MU}) \otimes \QQ$ that we will need to make use of.

The integral homology $\mathrm{H}_* (\mathrm{MU}; \ZZ)$ is generated by elements $b_i$ which are the images of the duals of $c_1 ^i$ under $\mathrm{H}_* (\CP^\infty; \ZZ) \to \mathrm{H}_* (BU; \ZZ) \cong \mathrm{H}_* (\mathrm{MU}; \ZZ)$. If we define the Newton polynomials in $b_i$ inductively by $N_1 (b) = b_1$ and $$N_n(b) = b_1 N_{n-1} (b) - b_2 N_{n-2} (b) + \dots + (-1)^{n-2} b_{n-1} N_1(b) + (-1)^{n-1} n b_n,$$ then $N_n (b)$ generates the group of primitive elements in $\mathrm{H}_{2n} (\mathrm{MU}; \ZZ)$. Furthermore, $N_n (b) \equiv (-1)^{n-1} n b_n$ modulo decomposables. As we will see in \Cref{DLRev}, there are convenient formulae for the action of the Dyer-Lashof operations on $N_n (b)$.

The homotopy $\pi_* (\mathrm{MU})$ of $\mathrm{MU}$ is generated by elements $x_i$ whose images under the Hurewicz map are $h(x_i) \equiv p b_i$ modulo decomposables when $i = p^n-1$ for some prime $p$ and $h(x_i) \equiv b_i$ modulo decomposables otherwise.

We may view the corbordism class of $\CP^n$ as an element $\bCPn$ of $\pi_{2n} (\mathrm{MU})$. Then, the $\bCPn$ do not generate $\pi_* (\mathrm{MU})$, though they are generators of $\pi_* (\mathrm{MU}) \otimes \QQ$. Under the isomorphism $\pi_* (\mathrm{MU}) \otimes \QQ \cong \mathrm{H}_* (\mathrm{MU}; \QQ)$ induced by the Hurewicz map, $\bCPn \equiv -(n+1) b_n$ modulo decomposables.

The logarithm of of the universal formal group $F$ on $\pi_* (\mathrm{MU})$ may be expressed in terms of the elements $\bCPn$: $$\ell_{F} (x) = \sum \frac{\left[\CP^{n-1}\right] x^n}{n}.$$


\subsection{When are the Dyer-Lashof operations defined?}

To obtain the precise bounds on $\mathbb{E}_n$ structures of \Cref{MainThm} and \Cref{MUThm}, we need to know when a Dyer-Lashof operation $Q^k$ is defined on an element $x \in \pi_n R$ for $R$ an $\mathbb{E}_n$-$\mathrm{H}$-algebra.

\begin{thm}[{\hspace{1sp}\cite[Theorems III.3.1 and III.3.3]{Hinf}}]\label{DLWhen}
	Let $R$ be an $\mathbb{E}_n$-$\mathrm{H}$-algebra. Then the operation $Q^s$ is defined on an element $x \in \pi_* R$ when $2s- \mathrm{deg}(x) \leq n - 1$; however, these operations only satisfy the expected properties (e.g. linearity, Cartan formula) when $2s - \mathrm{deg}(x) \leq n-2$.
\end{thm}

\subsection{Acknowledgements} The author would like to thank Tyler Lawson for explanations of his work \cite{BPtwo} and for introducing him to the world of homotopy theory. He would also like to thank his advisor, Haynes Miller, for helpful conversations about this work. He would further like to thank them for providing useful comments on a draft of this paper.

While revising this work, the author was supported by an NSF GRFP fellowship under Grant No. 1745302.

\section{Power operations in the homotopy of $\mathrm{MU}$}\label{POp} \label{primary}

\subsection{Statement of results}

Our goal in this section is to compute certain power operations in the homotopy of $\MU$ which will form the starting point of our proof that $\BP$ does not admit the structure of a $\mathbb{E}_{2(p^2+2)}$-ring.

We begin by recalling that the $\mathbb{H}_\infty ^2$-structure on $\mathrm{MU}$ equips the even $\mathrm{MU}$-cohomology of a space $X$ with a power operation $$P_{C_p} : \mathrm{MU}^{2*} (X) \to \mathrm{MU}^{2p*} (X \times B C_p).$$ Using the isomorphism $$\mathrm{MU}^* (B C_p) \cong \mathrm{MU}^* [[\alpha]]/[p]_F (\alpha),$$ we may view this power operation applied to $X=\ast$ a point as a map $$P_{C_p} : \mathrm{MU}^{2*} \to \mathrm{MU}^{2p*} [[\alpha]] / [p]_F (\alpha).$$

Let
\[r_* : \mathrm{MU}^* [[\alpha]] / [p]_F (\alpha) \to \mathrm{BP}^* [[\alpha]] / [p]_F (\alpha)\] denote the map induced by the Quillen idempotent. Our goal in this section will be to compute the composition of $r_* \circ P$ applied to certain elements of $\MU^{2*}$. We begin with the following piece of notation.

\begin{nota}
	Let $$\chi = \prod_{i=1}^{p-1} [i]_F (\alpha) \in \mathrm{MU}^{*} (BC_p) \cong \mathrm{MU}^* [[\alpha]] / [p]_F (\alpha)$$ denote the $\mathrm{MU}$-Euler class of the real reduced regular representation of $C_p$.
\end{nota}


\begin{thm}\label{POpThm}
    The follow equality holds modulo $\BP^*$-decomposables:
\begin{align} \label{eq:MU2}
    r_* \left( \chi^{p(p-1)} P_{C_p} \left(\left[\CP^{p(p-1)}\right]\right) \right) \equiv - v_3 \alpha^{p^3 - 1 - p(p-1)} + O(\alpha^{p^3})
\end{align}
\end{thm}


We may deduce the following corollary.

\begin{cor}\label{cor:JNcor}
    Suppose $f : \MU_{(p)} \to E$ is a map of $\mathbb{H}_{\infty}$-ring spectra satisfying:
    \begin{enumerate}
        \item $f$ factors through the Quillen idempotent $\MU_{(p)} \to \BP$.
        \item $f$ induces a Landweber exact $\MU_*$-module structure on $E_*$.
    \end{enumerate}
    Then the induced formal group on $E_*$ has height at most $2$, i.e. $v_2$ is invertible in $E_* / (p,v_1)$.
\end{cor}

This corollary is similar to \cite[Theorem 1.3]{JN}, which differs from it in the following respect: \cite[Theorem 1.3]{JN} shows the stronger result that $E_*$ is a $\QQ$-algebra, but only for primes $p \leq 13$.

\begin{proof}[Proof of \Cref{cor:JNcor}]
    By \cite[Theorem 1.3]{JN}, we may as well assume that $p > 2$.

    The map $\MU \to E$ automatically acquires an $\mathbb{H}^2 _{\infty}$-structure by \cite[Theorem 3.13]{JN}.
    Since $\left[\CP^{p(p-1)}\right]$ maps to zero under the Quillen idempotent and thus under $f_*$, it follows that \[P_{C_p} \left(\chi^{p(p-1)} \left[ \CP^{p(p-1)}\right]\right)\] maps to zero in $E_* [[\alpha]] / [p](\alpha)$.
    It follows from \Cref{POpThm} that
    \[r_* \left( \chi^{p(p-1)} P_{C_p} \left(\left[\CP^{p(p-1)}\right]\right) \right) \equiv -v_3 \alpha^{p^3 - 1 - p(p-1)} + O(\alpha^{p^3}) \mod (p,v_1,v_2).\]
    Since $[p] (\alpha) \equiv v_3 \alpha^{p^3} + O(\alpha^{p^3+1}) \mod (p,v_1,v_2)$, the above implies that $v_3 = 0 \in E_*/ (p,v_1,v_2)$
    Now, $v_3$ is regular in $E_* / (p,v_1,v_2)$ by Landweber exactness, so that we must have $E_* /(p,v_1,v_2) = 0$, as desired.
\end{proof}

We begin the proof of \Cref{POpThm} with a reduction. Since we are working modulo $\BP^*$-decomposables, the coefficient of $\alpha^{p^3-1-p(p-1)}$ in (\ref{eq:MU2}) is equivalent to a constant multiple of $v_3$. Moreover, this is the first term in (\ref{eq:MU2}) that can be nonzero modulo $\BP^*$-decomposables. It therefore suffices to show that and (\ref{eq:MU2}) holds after composing with the map $q : \mathrm{BP}^* \to \ZZ_p [v_3]/(v_3 ^2)$ that sends $v_3$ to $v_3$ and $v_i$ to $0$ for $i \neq 3$. Here, we let $v_i$ denote the $i$th Hazewinkel generator. In conclusion, to prove Theorem \hyperref[POpThm]{\ref*{POpThm}} it suffices to prove the following proposition.


\begin{prop}\label{PowerProp}
	There is an equality
    $$q \circ r_* \left( \chi^{p(p-1)} P_{C_p} \left(\left[\CP^{p(p-1)}\right]\right) \right) = -v_3 \alpha^{p^3-1 - p(p-1)}.$$
\end{prop}

In the appendix of \cite{BPtwo}, Lawson shows how this computation may be made internally to $\ZZ_p [v_3] / (v_3 ^2)$ and the induced formal group law by adapting a method of Johnson--Noel \cite{JN}. Since this formal group law is much simpler than the formal group law of $\mathrm{BP}$, the computation that we need to make simplifies dramatically and so becomes tractable.

\subsection{Proof of \Cref{PowerProp}}

    We begin by reviewing some basic facts about $MU$-power operations. This section is based on \cite[Appendix A]{BPtwo}.
    \begin{nota}
        We let $$\<p\>_F (x) = \frac{[p]_F (x)}{x}.$$
    \end{nota}

    \begin{fact} \label{prop:properties}
        The power operation $P_{C_p} : \mathrm{MU}^{2*} (X) \to \mathrm{MU}^{2p*} (X) [[\alpha]] / [p]_F (\alpha)$ satisfies the following properties:

	\begin{enumerate}
        \item $P_{C_p}(uv) = P_{C_p}(u)P_{C_p}(v)$

        \item $P_{C_p}(u) = u^p$ modulo $\alpha$

        \item $P_{C_p}(u + v) = P_{C_p}(u) + P_{C_p}(v)$ modulo $\<p\>_F (\alpha)$

        \item On the orientation class $x \in \widetilde{\mathrm{MU}}^{2} (\CP^\infty)$, $$P_{C_p}(x) = x\prod_{i=1}^{p-1} (x +_F [i]_F (\alpha)).$$
	\end{enumerate}
%
    \end{fact}
    
    \begin{nota}
        We let
	$$g(x,\alpha) = x\prod_{i=1}^{p-1} (x +_F [i]_F (\alpha)),$$
        viewed as an element of $\MU^* [[x,\alpha]]/[p]_F (\alpha)$, so that $P_{C_p} (x) = g(x,\alpha)$.
        Note that $\px g(0,\alpha) = \chi.$
    \end{nota}

    Applying \Cref{prop:properties} to the spaces $X = (\CP^\infty)^{\times n}$ for $n=0,1,2$ and the product map $(\CP^\infty)^{\times 2} \to \CP^\infty$, we obtain the following proposition:

    \begin{prop}
    The composite
        $$\Psi : \mathrm{MU}^* \to \mathrm{MU}^* [[\alpha]] / \<p\>_F (\alpha)$$
        of $P_{C_p}$ with the quotient map $\mathrm{MU}^* [[\alpha]] / [p]_F (\alpha) \to \mathrm{MU}^* [[\alpha]] / \<p\>_F (\alpha)$ is a ring homomorphism.
        Moreover, the power series $g(x,\alpha)$ defines an isogeny $F \to \Psi^* F$.
    \end{prop}

    Let $\omega \in \ZZ_p$ denote a $(p-1)$st root of unity. We will find it convenient to express $g(x,\alpha)$ and $\chi$ in terms of $[\omega^i]_{F} (\alpha)$ instead of $[i]_{F} (\alpha)$, where $i = 1, \dots, p-1$ on both sides. This is because we will eventually replace $F$ with a $p$-typical formal group law $G$, and for any $p$-typical $G$ we have the simple formula $[\omega^i]_{G} (x) = \omega^i x$.

    To make sense of this, we must base change to the $p$-completion $\mathrm{MU}^* _{p}$.
    When base changed to $\MU^* _p$, the formal group law $F$ admits the structure of a formal $\ZZ_p$-module.
    In particular, if we let $\omega \in \ZZ_p$ denote a primitive $(p-1)$st root of unity, there are endomorphisms $[\omega^i]_F (x)$ of $F$.
    Since $\omega^1, \dots, \omega^{p-1}$ form a set of representatives for $1, \dots, p-1$ modulo $p$, we obtain the following lemma:
    \begin{lemma}
        There are equalities
	\begin{align*}
    \chi \equiv \prod_{i=1} ^{p-1} [\omega^i]_F (\alpha) \mod [p]_F (\alpha)
	\end{align*}
	and 
	\begin{align*}
    g(x,\alpha) \equiv x\prod_{i=1}^{p-1} (x +_F [\omega^i]_F (\alpha)) \mod [p]_F (\alpha).
	\end{align*}
    \end{lemma}

    Since $\mathrm{MU}^*$ and $\mathrm{MU}^* [[\alpha]] / \<p\>_F (\alpha)$ are torsion-free, $F$ and $\Psi^* F$ admit logarithms $$\ell_F (x) = \sum \frac{\left[\CP^{n-1}\right] x^n}{n}$$ and $$\ell_{\Psi^* F} (x) = \sum \frac{\Psi\left(\left[\CP^{n-1}\right]\right) x^n}{n}.$$

    This implies that we may compute $\Psi(\bCPn)$ as the coefficient of $x^n$ in the derivative $\ell' _{\Psi^* F} (x)$ of $\ell_{\Psi^* F} (x)$ with respect to $x$. We will now describe a method for computing these coefficients. We begin with a lemma.

    \begin{lemma} \label{lem:chi-alpha}
        Let $R^*$ denote a nonzero graded torsion-free ring and let $r : \MU^* _p \to R^*$ denote a map classifying a formal group law $G$ over $R^*$. Then 
        \[ r(\chi) = \prod_{i=1}^{p-1} [\omega^i]_{G} (\alpha)\]
        may be written as $u \alpha^{p-1}$, where $u$ is a unit in $R^*$. Moreover, $\alpha$ is not a zero divisor in $R^* [[\alpha]]/\<p\>_{G} (\alpha)$, so neither is $\chi$.
    \end{lemma}

    \begin{proof}
        We have
        \begin{align*}
            r(\chi) &= \prod_{i=1} ^{p-1} [\omega^{i}]_G (\alpha) \\
            &= \prod_{i=1} ^{p-1} (\omega^i \alpha + O(\alpha^2)) \\
            &= \alpha^{p-1} (-1 + O(\alpha)),
        \end{align*}
        which implies that $r(\chi) = u\cdot \alpha^{p-1}$ for a unit $u$. 

        It remains to show that $\alpha$ is not a zero-divisor in $R^* [[\alpha]]/\<p\>_{G} (\alpha)$.
        Suppose that $\alpha \cdot f(\alpha) = g(\alpha) \cdot \<p\>_G (\alpha)$. We wish to show that $\alpha$ must divide $g(\alpha)$, or in other words that $g(\alpha)$ has trivial constant term. But this follows from the fact that $\<p\>_{G} (\alpha)$ has constant term $p$, which is not a zero divisor in $R^*$.
    \end{proof}

    \begin{defin}
        We fix an arbitrary lift $\Psi(\bCPn) \in \mathrm{MU}^* [[\alpha]]$ of $\Psi(\bCPn) \in \MU^* [[\alpha]]/\<p\>_{F} (\alpha)$. This determines a lift of $\ell_{\Psi^* F} (x)$ to $\mathrm{MU}^* [[x,\alpha]]$.

        We also fix a lift of $g(x, \alpha)$ to $\mathrm{MU}^* _p [[x,\alpha]]$. Then $\px g(0,\alpha)$ is a lift of $\chi$ to $\mathrm{MU}^* _p [[\alpha]]$.
    \end{defin}
    
    \begin{nota}\label{ntn:k}
        Define $k(y,\alpha)$ by $g(\chi y,\alpha) = \chi^{2} k(y,\alpha)$. Then $k(y,\alpha)$ has leading term $y$, so we may let $k^{-1} (y,\alpha)$ denote a composition inverse in the $y$ variable.

        Moreover, let $\ell' _{F} (x) = \px \ell_{F} (x)$, $\ell' _{\Psi^* F} (x) = \px \ell_{\Psi^* F} (x)$, $k'(y, \alpha) = \py k(y,\alpha)$ and $(k^{-1})' (y,\alpha) = \py k^{-1} (y,\alpha)$.
    \end{nota}

    \begin{prop}\label{prop:Comp<p>}
        Let $f_n (\alpha)$ denote the coefficient of $y^n$ in 
        \[ \ell'_{F} (\chi k^{-1} (y,\alpha)) \cdot (k^{-1})' (y,\alpha).\]
        Then
        \[\Psi(\bCPn) \chi^{2n} \equiv f_n (\alpha) \mod \<p\>_{F} (\alpha).\]
    \end{prop}

    \begin{proof}
	Applying $\py$ to the equation
    $$g(x,\alpha) +_{\Psi^* F} g(y,\alpha) \equiv g(x +_F y, \alpha) \mod \< p \> (\alpha)$$
    and evaluating at $y=0$, we obtain the equation
	\begin{align*}
        \frac{g'(0,\alpha)}{(\ell_{\Psi^* F})' (g(x,\alpha))} \equiv \frac{g'(x,\alpha)}{(\ell_F)' (x)} \mod \<p\>_F (\alpha).
	\end{align*}
        This implies that $$g'(x,\alpha) \cdot (\ell_{\Psi^* F})' (g(x,\alpha)) = \chi \cdot (\ell_F)' (x) + h(x,\alpha) \cdot \<p\>_{F} (\alpha)$$ for some $h(x,\alpha) \in \mathrm{MU}^* _p [[x,\alpha]]$. In the above equation, we have used the fact that $\chi = g'(0,\alpha)$. Plugging in $x=0$, we find that $h(0,\alpha) = 0$, so that $h(x,\alpha) = x \widetilde{h}(x,\alpha)$ for some $\widetilde{h}(x,\alpha) \in \mathrm{MU}^* _p [[x,\alpha]]$.

        Next, we make the substitution $x = \chi y$ and write $g(\chi y, \alpha) = \chi^2 k(y,\alpha)$ as in \Cref{ntn:k}. Plugging in our substitution, we obtain
	\begin{align*}
        \chi \cdot k'(y,\alpha) \cdot (\ell_{\Psi^* F})' (\chi^2 k(y,\alpha)) &= \chi \cdot (\ell_F)' (\chi y) + h(\chi y, \alpha) \cdot \<p\>_{F} (\alpha) \\
        &= \chi \cdot (\ell_F)' (\chi y) + \chi y \cdot \widetilde{h}(\chi y, \alpha) \cdot \<p\>_{F} (\alpha).
	\end{align*}
        Substituting $y = k^{-1} (z,\alpha)$, applying the chain rule and dividing by $\chi$ (which is valid by \Cref{lem:chi-alpha}), we obtain
	\begin{align*}
    (\ell_{\Psi^* F})' (\chi^2 z) = (\ell_F)' (\chi k^{-1} (z, \alpha)) \cdot (k^{-1})' (z,\alpha) + k^{-1} (z,\alpha)\cdot \widetilde{h} (\chi k^{-1} (z,\alpha),\alpha) \cdot \<p\>_F (\alpha).
	\end{align*}
	Taking coefficients of $z^n$ on both sides, we find that
	\begin{align*}
		\Psi(\CP^n) \chi^{2n} = f_n (\alpha) + \widetilde{h}_n (\alpha) \cdot \<p\>_F (\alpha)
	\end{align*}
	for some $\widetilde{h}_n (\alpha) \in \mathrm{MU}^* _p [[\alpha]]$, as desired.
    \end{proof}
    

    Finally, to compute $P_{C_p} (\bCPn)$, we have the following proposition.
    
       \begin{prop} \label{prop:method-MU}
        There exists a unique polynomial $h_n (\alpha) \in \MU^* [\alpha]$ of degree $2n(p-1)$ with the property that
        $$f_n (\alpha) - h_n (\alpha) \cdot \<p\>_F (\alpha) \equiv \chi^{2n} \bCPn^p \mod \alpha^{2n(p-1)+1}.$$
        Furthermore,
        \[ P_{C_p} (\bCPn) \equiv \chi^{-2n} (f_{n} (\alpha) - h_n (\alpha) \cdot \<p\>_F (\alpha)) \mod [p]_{F} (\alpha).\]
    \end{prop}

    \begin{proof}
        By \Cref{prop:Comp<p>}, \[f_n (\alpha) \equiv \chi^{2n} \Psi(\bCPn) \equiv \chi^{2n} P_{C_p} (\bCPn) \mod \<p\>_{F} (\alpha).\] By \Cref{prop:properties}(2), this implies that \[f_n (\alpha) \equiv \bCPn^p \mod (\<p\>_F (\alpha), \chi^{2n}\alpha).\]
        Combining the above with \Cref{lem:chi-alpha}, we find that $h_n (\alpha)$ exists. Uniqueness follows from the fact that the constant term $p$ of $\<p\>_F (\alpha)$ is not a zero divisor in $\MU^* _p$.

        In particular, we find that $f_n (\alpha) - h_n (\alpha) \cdot \<p\>_F (\alpha)$ is divisible by $\chi^{2n}$ and that
        \[\chi^{-2n} (f_n (\alpha) - h_n (\alpha) \cdot \<p\>_F (\alpha)) \equiv \bCPn^p \equiv P_{C_p} (\bCPn) \mod \alpha\]
        and
        \[\chi^{-2n} (f_n (\alpha) - h_n (\alpha) \cdot \<p\>_F (\alpha)) \equiv \Psi (\bCPn) \equiv P_{C_p} (\bCPn) \mod \<p\>_F (\alpha).\]
        Again using the fact that $p$ is not a zero divisor in $\MU^* _p$, this implies that 
        \[\chi^{-2n} (f_n (\alpha) - h_n (\alpha) \cdot \<p\>_F (\alpha)) \equiv P_{C_p} (\bCPn) \mod [p]_F (\alpha),\]
        as desired.
    \end{proof}

	Suppose now that we are given a graded torsion-free ring $R^*$ and a homomorphism $r : \mathrm{MU}^* _p \to R^*$ classifying a formal group law $G$ over $R^*$.
    Then we may define $\chi_{G}$, $g_{G} (x,\alpha)$, $k_G (x, \alpha)$, $k^{-1} _G (x, \alpha)$ and $f_n ^{G} (\alpha)$ as above, using the formal group law $G$ on $R^*$ in place of the formal group law $F$ over $\MU^*$.

    \begin{prop} \label{prop:method-R}
        Let $R^*$ denote a graded torsion-free ring, and let $r : \MU^* _p \to R^*$ classify a formal group law $G$ over $R^*$. Then there exists a unique polynomial $h^G _n (\alpha) \in R^* [\alpha]$ of degree $2n(p-1)$ with the property that
        \[ f^G _n (\alpha) - h^G _n (\alpha) \cdot \<p\>_G (\alpha) \equiv \chi^{2n} \bCPn^p \mod \alpha^{2n(p-1)+1}.\]
        Moreover,
        \[ r(P_{C_p} (\bCPn)) \equiv \chi^{-2n} (f^G _{n} (\alpha) - h^G _n (\alpha) \cdot \<p\>_G (\alpha)) \mod [p]_{G} (\alpha).\]
    \end{prop}

    \begin{proof}
        The first part follows exactly as in the proof of \Cref{prop:method-MU}.

        For the second part, we note that $f^G _n (\alpha) = r(f_n (\alpha))$ by the definitions, and that $h^G _n (\alpha) = r (h_n (\alpha))$ by uniqueness. The second part then follows from \Cref{prop:method-MU}.
    \end{proof}


%

    \begin{prop} \label{prop:bigcalc1}
        Consider the map $q \circ r_* : \mathrm{MU}^* _p \to \ZZ_p [v_3] / (v_3 ^2)$ and its induced formal group law $G = (q \circ r_*)^* F$. Then the following hold:
        \begin{enumerate}
            \item $\displaystyle \ell_G (x) = x + \frac{v_3}{p} x^{p^3}.$
            \item $\displaystyle x +_G y = x+y+\frac{v_3}{p}(x^{p^3} + y^{p^3} - (x+y)^{p^3}).$
            \item $\displaystyle[p]_G (\alpha) = p\alpha - (p^{p^3-1}-1) v_3 \alpha^{p^3},$ so that $\displaystyle \<p\>_G (\alpha) = p - (p^{p^3-1}-1) v_3 \alpha^{p^3-1}.$
            \item $\displaystyle \chi_G = \prod_{i=1} ^{p-1} \omega^i \alpha = -\alpha^{p-1}.$
            \item $g_G (x, \alpha) \equiv \chi x + x^p + O(x^{p^2}) \mod [p]_G (\alpha).$
            \item $k_G (y, \alpha) \equiv y + \chi^{p-2} y^p + O(y^{p^2}) \mod [p]_G (\alpha).$
            \item $\displaystyle k_G^{-1} (y,\alpha) = y + \sum_{n=1}^p (-1)^n \frac{{np \choose n}}{n(p-1) + 1} \chi^{n(p-2)} y^{n(p-1) + 1} + O(y^{p^2}).$
            \item $\displaystyle f^G _{i(p-1)} (\alpha) = (-1)^i {ip \choose i}{\chi^{i(p-2)}}$ for $1 \leq i \leq p$.
            \item $\displaystyle h^G _{i(p-1)} (\alpha) = (-1)^i \frac{{ip \choose i}}{p}{\chi^{i(p-2)}}$ for $1 \leq i \leq p$.
        \end{enumerate}
    \end{prop}
    \begin{proof}
        Part (1) follows from the formula for the logarithm of the universal $p$-typical formal group law \cite[Appendix A2]{GreenBook}. Recall that we are using the Hazewinkel $v_i$s. Parts (2) and (3) follow in a straightforward way from part (1).
        To establish part (4), we note that, since the formal group law $G$ is $p$-typical, $[\omega^i]_G (x) = \omega^i x$. Therefore
	\begin{align*}
		\chi_G = \prod_{i=1} ^{p-1} \omega^i \alpha = -\alpha^{p-1},
	\end{align*}
    since $p$ is odd. Moreover, we have
	\begin{align*}
		g_G(x,\alpha) = x \prod_{i=1} ^{p-1} (x +_G (\omega^i \alpha)).
	\end{align*}
%
	We then compute
	\begin{align*}
g_G (x, \alpha) &= x\prod_{i=1}^{p-1} (x +_G (\omega^i \alpha)) \\
	&= x \prod_{i=1}^{p-1} (x+ \omega^i \alpha) \left[ 1 + \frac{v_3}{p} \sum_{j=1}^{p-1} \frac{x^{p^3} + (\omega^j \alpha)^{p^3} - (x+\omega^j \alpha)^{p^3}}{x+\omega^j \alpha} \right] \\
		&\equiv x\prod_{i=1} ^{p-1} (x+\omega^i \alpha) + O(x^{p^2}) \mod [p]_G (\alpha) \\
	&= x(x^{p-1} - \alpha^{p-1}) + O(x^{p^2}) \\
	&= \chi x + x^p + O(x^{p^2}),
	\end{align*}
        where we have used the fact that $pv_3 \alpha = 0$ modulo $[p]_G (\alpha)$. This establishes part (5). Part (6) follows immediately from the defining equation $\chi^{2} k_G(y,\alpha) = g_G(\chi y, \alpha)$.

        To deduce part (7), we apply Lagrange inversion to part (6). Since $$(\ell_{G})' (x) = 1 + O(x^{p^3-1}),$$ we deduce that \[(\ell_G)' (\chi k^{-1} (y, \alpha) (k^{-1})' (y,\alpha) = (k^{-1})' (y,\alpha) + O(y^{p^3-1}),\] so we may read off (8) from (7).

        Finally, (9) follows from (8) and the fact that $\bCPn^p = 0$ in $\ZZ[v_3]/v_3 ^2$.
    \end{proof}

%
    \begin{cor} \label{cor:computefinal}
        Given $1 \leq i \leq p$, there is an equality
        \[q \circ r_* \left( \chi^{i(p-1)} P_{C_p} \left(\left[\CP^{i(p-1)}\right]\right) \right) \equiv - \frac{{ip \choose i}}{p} v_3 \alpha^{p^3-1-i(p-1)} \mod [p]_G (\alpha).\]
    \end{cor}
    \begin{proof}
        Using \Cref{prop:method-R} and \Cref{prop:bigcalc1}, we compute:
	\begin{align*}
    q \circ r_* \left( \chi^{i(p-1)} P_{C_p}\left(\left[\CP^{i(p-1)}\right]\right) \right) &\equiv \chi_G ^{-i(p-1)} \cdot (f^G _{i(p-1)} (\alpha) - h^G _{i(p-1)} (\alpha) \cdot \<p\>_G (\alpha)) \\
		&\equiv \chi_G ^{-i(p-1)} \cdot (-h^G _{i(p-1)} (\alpha)) \cdot (- (p^{p^3-1}-1) v_3 \alpha^{p^3-1}) \\
		&\equiv -h^G _{i(p-1)}(\alpha) v_3 \alpha^{p^3-1-i(p-1)^2} \\
        &\equiv (-1)^{i+1} \chi_G ^{i(p-2)} \frac{{ip \choose i}}{p} v_3 \alpha^{p^3-1-i(p-1)^2} \\
        &\equiv - \frac{{ip \choose i}}{p} v_3 \alpha^{p^3-1-i(p-1)} \mod [p]_G (\alpha),
	\end{align*}
	where we have used the fact that $pv_3 \alpha = 0$ modulo $[p]_G (\alpha)$.
    \end{proof}

    \begin{proof}[Proof of \Cref{PowerProp}]
        Applying the congruence
        $\frac{{p^2 \choose p}}{p} \equiv 1 \mod p$ to \Cref{cor:computefinal}, we deduce that
	\begin{align*}
    q \circ r_* \left( \chi^{p(p-1)} P_{C_p} \left(\left[\CP^{p(p-1)}\right]\right) \right) \equiv - v_3 \alpha^{p^3-1-p(p-1)} \mod [p]_G (\alpha),
	\end{align*}
    as desired.
    \end{proof}

\begin{rmk}
	Lawson has informed the author that Zeshen Gu has independently worked on computations similar to those in this section.
\end{rmk}


\section{A Dyer-Lashof operation in the $\mathrm{MU}$-dual Steenrod algebra}\label{DLMUStSec}

In this section, we apply Theorem \hyperref[POpThm]{\ref*{POpThm}} to compute certain Dyer-Lashof operations in the $\MU$-dual Steenrod algebra $\pi_* (\mathrm{H} \wedge_{\mathrm{MU}} \mathrm{H})$. We begin by determining the structure of $\pi_* (\mathrm{H} \wedge_{\mathrm{MU}} \mathrm{H})$ as an algebra.

\begin{prop}
    The algebra $\pi_* (\mathrm{H} \wedge_{\mathrm{MU}} \mathrm{H})$ is isomorphic to an exterior algebra $\Lambda_{\FF_p} (\tau_i) \otimes \Lambda_{\FF_p} (\sigma m_i \, \vert \, i \neq p^k-1)$ on classes $\tau_i$ for $i \geq 0$ and $\sigma m_i$ for $i \geq 1$. The degrees of these classes are $\abs{\tau_i} = 2p^i-1$ and $\abs{\sigma m_i} = 2i+1$.

    The natural map $\mathrm{H} \wedge \mathrm{H} \to \mathrm{H} \wedge_{\mathrm{MU}} \mathrm{H}$, upon taking homotopy, induces a map $$\Lambda_{\FF_p} (\tau_i) \otimes \FF_p [\xi_i] \to \Lambda_{\FF_p} (\tau_i) \otimes \Lambda_{\FF_p} (\sigma m_i \, \vert \, i \neq p^k-1)$$ sending $\tau_i$ to $\tau_i$ and $\xi_i$ to $0$.
\end{prop}

\begin{proof}
  It is standard to compute
  \[\Tor_{*,*} ^{\mathrm{H}_* \mathrm{MU}} (\mathrm{H}_*, \mathrm{H}_*\mathrm{H}) \cong \Lambda_{\FF_p} (\tau_i) \otimes \Lambda_{\FF_p} (\sigma m_i \, \vert \, i \neq p^k-1),\]
  so that the K\"unneth spectral sequence
  $$\Tor_{*,*} ^{\mathrm{H}_* \mathrm{MU}} (\mathrm{H}_*, \mathrm{H}_*\mathrm{H}) \Rightarrow \pi_* \left(\left(\mathrm{H} \wedge \H\right) \wedge_{\H \wedge \mathrm{MU}} \left(\H \wedge \mathrm{H}\right)\right) = \pi_* (\H \wedge_{\MU} \H)$$
  degenerates on the $\mathrm{E}_2$-page for degree reasons.
  There are no possible extension problems, and the computation of the map
  \[\pi_* (\H \wedge \H) \to \Tor_{*,*} ^{\mathrm{H}_* \mathrm{MU}} (\mathrm{H}_*, \mathrm{H}_*\mathrm{H}) \to \pi_* (\H \wedge_{\MU} \H)\]
  follows from the degeneration.
\end{proof}

\begin{rmk}
    Note that there is a second K\"unneth spectral sequence
    $$\Tor_{*,*} ^{\pi_* \mathrm{MU}} (\mathrm{H}_*, \mathrm{H}_*) \Rightarrow \pi_* (\mathrm{H} \wedge_{\mathrm{MU}} \mathrm{H}),$$
    which gives an alternative description of $\pi_* (\mathrm{H} \wedge_{\mathrm{MU}} \mathrm{H})$ as $\Lambda_{\FF_p} (\tau_0, \sigma x_i)$. Furthermore, Lawson \cite[Section 3.3]{BPtwo} shows that for $x \in \pi_n R$ where $n \geq 1$, there is a map $\widetilde{\mathrm{H}}_{*} (SL_1 (R)) \to \pi_{* + 1} (\mathrm{H} \wedge_{R} \mathrm{H})$ which sends the Hurewicz image of $x \in \pi_n R \cong \pi_n SL_1 (R)$ to a distinguished choice of $\sigma x$.

	Furthermore, this map $\sigma : \pi_n R \to \pi_{n+1} (\mathrm{H} \wedge_R \mathrm{H})$ annihilates decomposables. Whenever we write $\sigma x$ for $x \in \pi_n R$, we will be referring to this distinguished choice of $\sigma x$.
	
\end{rmk}

We are now able to state the main theorem of this section.

\begin{thm}\label{DLMUSt}
	In $\pi_* (\mathrm{H} \wedge_{\mathrm{MU}} \mathrm{H})$, we have
    $$ Q^{p^2+1} \left(\sigma \left[\CP^{p(p-1)}\right]\right) = - \sigma v_3.$$
\end{thm}

This follows immediately from Theorem \hyperref[POpThm]{\ref*{POpThm}} and the following theorem:

\begin{thm} \label{thm:convert}
    Let $y \in \pi_{2n} \mathrm{MU}$ and suppose that $$\chi^nP(y) = \sum_{i=0} ^\infty c_i \alpha^i$$ for some elements $c_i \in \pi_{2(n + i)} \mathrm{MU}$. Then the action of the Dyer-Lashof operations on $\pi_* (\mathrm{H} \wedge_{\mathrm{MU}} \mathrm{H})$ are determined by the equation
	\begin{align*}
    Q^k (\sigma y) = (-1)^k \sigma c_{k(p-1)}.
	\end{align*}
\end{thm}

Our proof of this theorem will follow \cite[Sections 3 and 4]{BPtwo} rather closely.
The idea will be to relate the power operation $P_{C_p}$ to the action of the multiplicative Dyer-Lashof operations on the homology of the $\Omega$-spectrum of $\MU$, and to relate this in turn to the Dyer-Lashof operations on $\pi_* (\H \wedge_\MU \H)$.

First, we need to introduce some notation from \cite[Section 4]{BPtwo}.

\begin{nota}
    We let \[\MU_{n} = \Omega^{\infty} \Sigma^{n} \MU\] denote the spaces in the $\Omega$-spectrum for $\MU$.

    Since $\MU$ is a ring spectrum, the homology of the spaces $\MU_n$ is equipped with two products, making $\H_* (\MU_\bullet)$ into a Hopf ring.
    We denote the additive one, coming from the infinite loop space structure on $\MU_n$, by
    \[- \# - : \H_* (\MU_n) \otimes \H_* (\MU_n) \to \H_* (\MU_n),\]
    and the multiplicative one, coming from the multiplication on $\MU$, by
    \[- \circ - : \H_* (\MU_n) \otimes \H_* (\MU_m) \to \H_* (\MU_{n+m}).\]

    Since $\MU$ is an $\mathbb{E}_\infty$-ring spectrum, $\MU_0$ is equipped with the structure of an $\mathbb{E}_\infty$-ring space. Its homology is therefore equipped with two actions of he Dyer-Lashof operations, an additive action coming the infinite loop space structure on $\MU_0$, and a multiplicative one coming from the $\mathbb{E}_\infty$-multiplication on $\MU$.

    We denote the additive operations by
    \[Q^k : \H_n (\MU_0) \to \H_{n+2(p-1)k} (\MU_0) \]
    and the multiplicative operations by
    \[\Qh^k : \H_n (\MU_0) \to \H_{n+2(p-1)k} (\MU_0). \]
\end{nota}

\begin{defin}
    The $\mathbb{H}_\infty ^2$-algebra structure on $\MU$ implies the existence of based maps
    \[\MU_{2n} \wedge (B\Sigma_p)_+ \to \MU_{2pn}\]
    representing the power operation
    \[P : \MU^{2n} (X) \to \MU^{2pn} (X \times B\Sigma_p).\]
    We let
    \[\Q : \H_* (\MU_{2n}) \to \H_* (\MU_{2pn}) \widehat{\otimes} \H^* (B\Sigma_p)\]
    denote the adjoint to the map
    \[H_* (\MU_{2n}) \otimes \H_* (B\Sigma_p) \to \H_* (\MU_{2pn}) \]
    induced by the above map of spaces.
\end{defin}

Multiplicativity of $P$ implies the following:

\begin{prop}[{\hspace{1sp}\cite[Proposition 4.2.2]{BPtwo}}]
    The operation $\Q$ preserves the $\circ$-product: $\Q(x) \circ \Q(y) = \Q(x \circ y)$.
\end{prop}

\begin{nota}
    Let $b_i \in \H_{2i} (\MU_2)$ denote the image under the orientation map $\CP^{\infty} \to \MU_2$ of the class in $\H_{2i} (\CP^\infty)$ dual to $c_1 ^i$. We let $b(s) = \sum_{i=1} ^\infty b_i s^i$, viewed as a formal power series in $s$.
\end{nota}

\begin{rmk} \label{rmk:b1-susp}
    Since $b_1$ is the fundamental class of the unit map $S^2 \to \MU_2$, $-\circ b_1 : \H_*(\MU_{2n}) \to \H_*(\MU_{2n+2})$ corresponds to suspension.
\end{rmk}

\begin{nota}
    Given a homotopy element $x \in \pi_{2n} (\MU)$, we let $[x] \in \H_0 (\MU_{2n})$ denote the image of the corresponding class in $\pi_{0} (\MU_{2n})$ under the Hurewicz map.

    It follows from \Cref{rmk:b1-susp} that $[x] \circ b_1 ^{\circ n} \in \H_{2n} (\MU_0)$ is the image of $x$, viewed as an element of $\pi_{2n} (\MU_0)$, under the Hurewicz map.
\end{nota}

\begin{defin}
    Given a based space $X$, there is a natural map
    \[\Lambda : \MU^{2n} (X) = [X, \MU_{2n}] \to \Hom(\H_* (X), \H_* (\MU_{2n})) = \H_* (\MU_{2n}) \widehat{\otimes} \H^* (X)\]
    which sends a homotopy class of map to its induced map on homology.
\end{defin}

The groups $\H_* (\MU_{2n}) \widehat{\otimes} \H^* (X)$ are equipped with products $\#$ and $\circ$, each induced by the corresponding product in $\H_* (\MU_{2n})$ and the cup product in $\H^* (X)$.

\begin{prop}[{\hspace{1sp}\cite[Propositions 3.2.3 and 4.2.3]{BPtwo}}]
    The map $\Lambda$ satisfies the following properties:
    \begin{itemize}
        \item $\Lambda(x+y) = \Lambda (x) \# \Lambda(y)$
        \item $\Lambda(xy) = \Lambda (x) \circ \Lambda (y)$
        \item $\Lambda([c]) = [c] \otimes 1$
        \item $(\Q \otimes 1)(\Lambda(x)) = \Lambda(P(x))$.
    \end{itemize}
\end{prop}

\begin{nota}
    Recall that
    \[\H^* (BC_p) \cong \FF_p [w] \otimes \Lambda_{\FF_p} (v),\]
    where $\abs{v} = 1$, $\abs{w} = 2$, and $w$ is the image of the generator $c_1$ of $H^2 (\CP^\infty)$ under
    the map on cohomology induced by the map
    \[BC_p \to \CP^\infty\]
    corresponding to the standard inclusion $C_p \hookrightarrow S^1$.
    Furthermore,
    \[\H^* (B\Sigma_p) \cong \FF_p [u] \otimes \Lambda_{\FF_p} (z),\]
    where, when pulled back to $BC_p$, $u = w^{p-1}$ and $z = v w^{p-2}$.

    Furthermore, we let $\beta_n$ be dual to $u^n$ in $\mathrm{H}^* (B \Sigma_p) \cong \FF_p [u] \otimes \Lambda_{\FF_p} [z]$ and $\gamma_n$ be dual to $u^{n-1} z$.
\end{nota}

\begin{rmk}[{\hspace{1sp}\cite[Remark 4.2.4]{BPtwo}}]
    Recall that $\MU^* (BC_p) \cong \MU^* [[\alpha]]/[p]_F (\alpha)$ with $\alpha \in \MU^2 (BC_p)$. The class $\alpha$ satisfies the equation
    \[\Lambda(\alpha) = b(w).\]
\end{rmk}

\begin{nota}
  Let $P_j$ be the homology operation dual to the Steenrod power $P^j$.
\end{nota}

\begin{lemma}
	For a space $X$ with $p$th extended power $D_p (X)$, the composite diagonal map on mod $p$ homology
  $$\mathrm{H}_* (X) \otimes \mathrm{H}_* (B \Sigma_p) \xrightarrow{\sim} \mathrm{H}_* (X \times B\Sigma_p) \to \mathrm{H}_* (D_p (X))$$
	is given by
  $$ x \otimes \beta_n \mapsto (-1)^{n} \sum_{j \geq 0} Q^{j+n} (P_j x)$$
	and
  $$ x \otimes \gamma_n \mapsto (-1)^{n+\abs{x}} \left( \sum_{j \geq 0} \beta Q^{j+n} (P_j x) - \sum_{j \geq 0} Q^{j+n} (P_j \beta x) \right)$$
\end{lemma}

\begin{proof}
  This follows from the definition of the Dyer-Lashof operations (cf. \cite[Definition 2.2]{MaySt}) and \cite[Proposition 9.1]{MaySt}.
    Note that an extra sign is introduced in the latter equation due to the fact that we have written the $B\Sigma_p$-action on the right and not the left.
    See also \cite[Proposition 6]{HomOpRev}.
\end{proof}

The following corollary then follows from the definitions:

\begin{cor} \label{thm:Q-Q}
    Suppose that $x \in \H_{*} (\MU_0)$. Then:
    \[\Q(x) = \sum_{n,j} (-1)^n \Qh^{j+n} (P_j x) u^n + (-1)^{n+\abs{x}}\left(\beta \Qh^{j+n} (P_j x) - \Qh^{j+n} (P_i \beta x) \right)u^{n-1}z.\]
    In particular, if $x$ is in the image of the Hurewicz map, then
    \[\Q(x) = \sum_{n=0} ^\infty (-1)^n \Qh^{n} (x) u^n + (-1)^{n+\abs{x}} \beta \Qh^{n} (x) u^{n-1} z.\]
\end{cor}


\begin{prop}
    We have \[\Q (b_1) = b_1 \circ \Lambda(\chi).\]
\end{prop}

\begin{proof}
    This is the second to last equation in the proof of \cite[Proposition 4.3.1]{BPtwo}.
\end{proof}

\begin{prop} \label{prop:formula}
    Let $y \in \pi_{2n} \mathrm{MU}$ and suppose that $$\chi^nP(y) = \sum_{i=0} ^\infty c_i \alpha^i$$ for some elements $c_i \in \pi_{2(n + i)} \mathrm{MU}$.

    Then, modulo $\#$-decomposables and the $\circ$-ideal generated by $b_2, b_3, \dots$, we have
    \[\Qh^{k}([y] \circ b_1^{\circ n}) \equiv (-1)^k [c_{(p-1)k}] \circ b_1 ^{\circ (p-1)k}.\]
\end{prop}

\begin{proof}
    We have:
    \begin{align*}
        \Q([y] \circ b_1^{\circ n}) &= \Q([y]) \circ \Q(b_1)^{\circ n} \\
        &= \Lambda (P(y)) \circ \Lambda (\chi)^{\circ n} \\
        &= \Lambda (P(y) \chi ^n) \\
        &= \Lambda (\sum_{i=0} ^\infty c_i \alpha^i) \\
        &= \hash_{i=0} ^\infty [c_i] \circ b(w) ^{\circ i} \\
        &\equiv \sum_{i=0} ^{\infty} [c_i] \circ (b_1)^{i} \circ w^{i},
    \end{align*}
    where we view $\Q ([y] \circ b_1 ^{\circ n})$ as living inside of $\H_* (\MU_0) \widehat{\otimes} \H^* (BC_p)$ via the natural inclusion $\H^* (B\Sigma_p) \hookrightarrow \H^* (BC_p)$.

    The result now follows from \Cref{thm:Q-Q} and the fact that the operations $P_j$ and $P_j \beta$ vanish on the spherical class $[y] \circ b_1 ^{\circ n}$.
\end{proof}

\begin{prop}\label{shift}
Let $p$ be an odd prime. Then the multiplicative Dyer-Lashof operations in the Hopf ring of an $\mathbb{E}_\infty$-ring space satisfy the following identity whenever $y$ is in the homology of the path component of zero:
	$$\widehat{Q^s} ([1] \# y) \equiv [1] \# \widehat{Q^{s}} (y)$$
	modulo $\#$ and $\circ$ decomposables.
\end{prop}

We first prove a lemma.

\begin{lemma}\label{declemma}
	In the situation of \Cref{shift}, for any $x$ there exist elements $z_i$ for $0 < i < \abs{x}$ such that the additive Dyer-Lashof operations satisfy $$Q^s (x) = Q^s [1] \circ x + \sum Q^{s_i} [1] \circ z_i.$$ Therefore $Q^s (x)$ is $\circ$-decomposable for any $x$ and any $s > 0$.
\end{lemma}

\begin{proof}
    This follows from the formula $$Q^s [1] \circ x = \sum_i Q^{s+i} ([1] \circ P_i x)$$ of \cite[Proposition II.1.6]{CLM} by inducting on the degree of $x$.
\end{proof}

\begin{proof}[Proof of \Cref{shift}]
  Let $\epsilon$ denote the counit and $\Delta_i$ denote the $i$-fold coproduct.

  We apply the mixed Cartan formula \cite[Theorem II.2.5]{CLM}, which states that $$\widehat{Q^s} (x \# y) = \sum_{s_0+\dots + s_p = s} \sum \widehat{Q^{s_0} _0} (x_0 \otimes y_0) \# \dots \# \widehat{Q^{s_p} _p} (x_p \otimes y_p)$$ where $$\Delta_{p+1} (x \otimes y) = \sum (x_0 \otimes y_0) \otimes \dots \otimes (x_p \otimes y_p)$$ and where $$\widehat{Q^s _0} (x \otimes y) = \epsilon(x) \widehat{Q^s} (y),$$ $$\widehat{Q^{s} _p} (x \otimes y) = \epsilon(y) \widehat{Q^s} (x),$$ and for $0 < i < p$ we put $m_i = \frac{1}{p} {p \choose i}$ so that $$\widehat{Q^s _i} (x \otimes y) = [m_i] \circ \left( \sum Q^{j} (x_1 \circ \dots \circ x_i \circ y_1 \circ \dots \circ y_{p-i}) \right)$$ where $\Delta_i x = \sum x_1 \otimes \dots \otimes x_i$ and $\Delta_{p-i} y = \sum y_1 \otimes \dots \otimes y_{p-i}$.

	Applying this to the case that $x = [1]$ and $y$ is in the homology of the path component of zero, we first note that this is $\#$-decomposable unless all of but one of the terms lies in degree $0$, i.e. unless $y_i= [0]$ and $s_i = 0$ for all but one $i$.

  Using \Cref{declemma}, we further deduce that all of the terms with $s_i \neq 0$ for some $0 < i < p$ are $\circ$-decomposable. Finally, we note that $\widehat{Q^s _p} ([1] \otimes y) = \widehat{Q^s} ([1]) = 0$ for $s > 0$ by \cite[Lemma II.2.6]{CLM}, so that in fact the only term left is $$\widehat{Q^{s} _0} ([1] \otimes y) \# \widehat{Q^0 _1} ([1] \otimes [0]) \# \dots \# \widehat{Q^0 _p} ([1] \otimes [0]).$$
  Now, for $0 < i < p$ we have $\widehat{Q^0 _i} ([1] \otimes [0]) = [m_i] \circ Q^0 ([0]) = [m_i] \circ [0] = [0]$. Furthermore, we have $\widehat{Q^0 _p} ([1] \otimes [0]) = \epsilon([0]) \widehat{Q^0} ([1]) = [1]$ and $\widehat{Q^s _0} ([1] \otimes y) = \epsilon([1]) \widehat{Q^s} (y) = \widehat{Q^s} (y)$.
  It follows that
  \[\widehat{Q^{s} _0} ([1] \otimes y) \# \widehat{Q^0 _1} ([1] \otimes [0]) \# \dots \# \widehat{Q^0 _p} ([1] \otimes [0]) = \widehat{Q^{s}} y \# [1].\]

  All that remains is to show that the multiplicity of this term is one, i.e. that $$([1] \otimes y) \otimes ([1] \otimes [0]) \otimes \dots \otimes ([1] \otimes [0])$$ appears with coefficient one in $\Delta_{p+1} ([1] \otimes y)$.
  This follows from the fact that $\Delta_{p+1} ([1]) = [1] \otimes \dots \otimes [1]$ and that $x \otimes [0] \otimes \dots \otimes [0]$ appears in $\Delta_{p+1} (x)$ with coefficient one.
\end{proof}

We are now ready to prove \Cref{thm:convert}.

\begin{proof}[Proof of \Cref{thm:convert}]
  In \cite[Section 3.3]{BPtwo}, a suspension map $\sigma : \widetilde{\H}_* (SL_1 (\MU)) \to \pi_{*+1} (\H \wedge_{\MU} \H)$ is constructed. By the mod $p$ analogues of \cite[Corollary 3.3.6 \& Proposition 3.3.7]{BPtwo} (which are proved in exactly the same way as for $p=2$), this map commutes with the Dyer-Lashof operations and kills $\#$-decomposables, $\circ$-decomposables, and $b_i$ for $i > 1$.
  Applying $\sigma$ to \Cref{prop:formula} and \Cref{shift}, we obtain the desired result.
\end{proof}

Our next goal is to deduce Theorem \hyperref[MUThm]{\ref{MUThm}} from Theorem \hyperref[DLMUSt]{\ref*{DLMUSt}} by noting that the Dyer-Lashof operations exhibited therein are incompatible with the existence of a highly structured map $\mathrm{H} \wedge_{\mathrm{MU}} \mathrm{H} \to \mathrm{H} \wedge_{\mathrm{BP}} \mathrm{H}.$ We begin by showing that a highly structured map $\mathrm{MU} \to \mathrm{BP}$ would induce a (slightly less) highly structured map $\mathrm{H} \wedge_{\mathrm{MU}} \mathrm{H} \to \mathrm{H} \wedge_{\mathrm{BP}} \mathrm{H}$.

\begin{prop}\label{En}
	Let $R$ be an $\mathbb{E}_\infty$-ring and let $A \to B$ denote a map of $\mathbb{E}_n$-rings augmented over $R$. Then there exists a natural map $R \wedge_{A} R \to R \wedge_{B} R$ of $\mathbb{E}_{n-1}$-$(R \wedge R)$-algebras.
\end{prop}

\begin{proof}
    Let $\CC$ denote the $\infty$-category $\mathrm{Alg}^{\mathbb{E}_{n-1}} _{R}$ of $\mathbb{E}_{n-1}$-$R$-algebras, equipped with the symmetric monoidal structure induced by that of $\mathrm{Mod}_R$. Then the bar construction defines a functor $\mathrm{Bar} : \mathrm{Alg} (\CC)_{/R} \to \CC$ by \cite[Remark 5.2.2.19]{HA}. By \cite[Theorem 5.1.2.2]{HA}, $\Alg (\CC)$ is equivalent to $\mathrm{Alg}^{\mathbb{E}_n} _{R}$, so that $\mathrm{Bar}$ defines a functor from augmented $\mathbb{E}_n$-$R$-algebras to $\mathbb{E}_{n-1}$-$R$-algebras.

    Since the forgetful functor $\CC \to \mathrm{Mod}_R$ preserves sifted colimits by \cite[Proposition 3.2.3.1]{HA}, Bar is computed in $R$-modules and so $\mathrm{Bar}(-) \cong R \wedge_{-} R$ as functors into $R$-modules.

	This implies the existence of a natural map $R \wedge_{A \wedge R} R \to R \wedge_{B \wedge R} R$ of $\mathbb{E}_{n-1}$-$R$-modules. Applying the functor $- \wedge_{R} (R \wedge R)$ yields the desired map $R \wedge_A R \to R \wedge_B R$ of $\mathbb{E}_{n-1}$-$(R \wedge R)$-algebras.
\end{proof}

We are now ready to prove Theorem \hyperref[MUThm]{\ref*{MUThm}}.
The $p=2$ case of this theorem was sketched by Lawson in \cite[Remark 4.4.7]{BPtwo}.

\begin{proof}[Proof of \Cref{MUThm}]
	For the sake of simplicity of notation, we prove Theorem \hyperref[MUThm]{\ref*{MUThm}} for $\mathrm{BP}$. The proof for $\mathrm{BP}\<n\>$ with $n \geq 3$ is analogous. Taking the $p$-completion changes nothing because we are only using the mod $p$ homology.

	Begin by noting that the K\"unneth spectral sequence
  $$\Tor_{*,*} ^{\pi_* \mathrm{BP}} (\mathrm{H}_*, \mathrm{H}_*) \Rightarrow \pi_* (\mathrm{H} \wedge_{\mathrm{BP}} \mathrm{H})$$
  collapses at the $E^2$-term, so that there is an isomorphism
  $\pi_* (\mathrm{H} \wedge_{\mathrm{BP}} \mathrm{H}) \cong \Lambda_{\FF_p} (\sigma v_i)$.

	Suppose that there were a map of $\mathbb{E}_{2p+3}$-rings $\mathrm{MU} \to \mathrm{BP}$. By the naturality of Postnikov towers of $\mathbb{E}_{2p+3}$-rings, this is a map of $\mathbb{E}_{2p+3}$-algebras augmented over $\mathrm{H}$. Then Proposition \hyperref[En]{\ref*{En}} implies that this induces a map $\mathrm{H} \wedge_{\mathrm{MU}} \mathrm{H} \to \mathrm{H} \wedge_{\mathrm{BP}} \mathrm{H}$ of $\mathbb{E}_{2p+2}$-$(\mathrm{H} \wedge \mathrm{H})$-algebras. Forgetting the action of the left $\mathrm{H}$, we obtain a map of $\mathbb{E}_{2p+2}$-$\mathrm{H}$-algebras.
	
    Now, the induced map $\Lambda_{\FF_p} (\tau_0, \sigma x_i) \cong \mathrm{H} \wedge_{\mathrm{MU}} \mathrm{H} \to \mathrm{H} \wedge_{\mathrm{BP}} \mathrm{H} \cong \Lambda_{\FF_p} (\sigma v_k)$ sends $\sigma x_{p^k -1}$ to a nonzero multiple of $\sigma v_k$.
    On the other hand, \Cref{DLMUSt} implies that we have $Q^{p^2+1} \sigma x_{p(p-1)} = C \sigma x_{p^3-1}$ for some nonzero $C$.
    Since $\sigma x_{p(p-1)}$ goes to zero in $\Lambda_{\FF_p} (\sigma v_k)$ for degree reasons, this is a contradiction because the operation $Q^{p^2+1} \sigma x_{p(p-1)}$ is preserved by maps of $\mathbb{E}_{2p+2}$-$\mathrm{H}$-algebras by \Cref{DLWhen}.
	%
\end{proof}

\section{A secondary power operation in the dual Steenrod algebra}\label{secondary}

In this section, we define and compute a secondary power operation in the dual Steenrod algebra and deduce Theorem \hyperref[MainThm]{\ref*{MainThm}} from this computation. We make free use of the formalism of Toda brackets in categories enriched over pointed topological spaces developed in \cite[Section 2]{BPtwo}, including the juggling, additivity and Peterson-Stein formulae of \cite[Propositions 2.3.5 and 2.4.3]{BPtwo}.

\begin{nota}
	Given a set $S$ of graded formal variables, we let $\mathbb{P}_\mathrm{H} ^{n} (S)$ denote the free $\mathbb{E}_n$-$\mathrm{H}$-algebra on the wedge of spheres $\displaystyle \bigvee_{x \in S} S^{\abs{x}}$ and let $x \in \pi_{\abs{x}} \left( \mathbb{P}_\mathrm{H} ^{n} (S) \right)$ denote the homotopy element corresponding to the fundamental class $\iota_{\abs{x}} \in \pi_{\abs{x}} \left( S^{\abs{x}} \right)$.
\end{nota}

Let $x$ be a formal variable of degree $2(p-1)$ and let $\mathbb{P}_\mathrm{H} ^{2(p^2+2)} (x)$ denote the free $\mathbb{E}_{2(p^2+2)}$-$\mathrm{H}$-algebra on $x$. Then we will let $\DD$ denote the category $\left( \mathrm{Alg}_{\mathrm{H}} ^{\mathbb{E}_{2(p^2+2)},\mathrm{aug}} \right)_{\mathbb{P}_\mathrm{H} ^{2(p^2+2)} (x)/}$ of augmented $\mathbb{E}_{2(p^2+2)}$-$\mathrm{H}$-algebras under $\mathbb{P}_\mathrm{H} ^{2(p^2+2)} (x)$. We note that the initial object of $\mathcal{D}$ is $\mathbb{P}_\mathrm{H} ^{2(p^2+2)} (x)$, while the final object is $\H$. Every $\mathbb{E}_{2(p^2+2)}$-$\H$-algebra we will consider below is connective with $\pi_0 = \FF_p$, hence admits a unique augmentation to $\mathrm{H}$.

The category $\mathcal{C}$ is a topological category, so the category $\CC = \DD^{\pm}$ of possibly pointed or augmented objects \cite[Definition 2.2.2]{BPtwo} in this category is enriched over pointed topological spaces. The category $\CC$ consists of augmented objects of $\DD$, pointed objects of $\DD$, and objects of $\DD$ without a pointing or augmentation. Through casework, one is able to define pointed spaces of maps between these objects, making use of the pointings and augmentations in the expected way when present. We refer the reader to \cite[Definition 2.2.2]{BPtwo} for the somewhat lengthy details.

Whenever we take brackets in the below, it will be in the category $\CC$. Given a set of graded elements $S$, we always view $\mathbb{P}_\mathrm{H} ^{2(p^2+2)} (x, S)$ as an element of $\CC$ via the augmentation $\mathbb{P}_\mathrm{H} ^{2(p^2+2)} (x,S) \to \mathbb{P}_\mathrm{H} ^{2(p^2+2)} (x)$ sending $x$ to $x$ and all of the elements of $S$ to $0$.

\begin{nota}
In the following, we will make our computations in the exterior quotient $\Lambda_{\FF_p} (\tau_0, \tau_1, \dots)$ of the dual Steenrod algebra $\mathrm{H}_* \mathrm{H}$; we call this quotient $\mathcal{E}_*$.
\end{nota}

\subsection{Dyer-Lashof operations in $\mathrm{H}_* (\mathrm{MU})$ and $\mathrm{H}_* \mathrm{H}$}\label{DLRev} We will need to be able to compute Dyer-Lashof operations in $\mathrm{H}_* (\mathrm{MU})$ and $\mathrm{H}_* \mathrm{H}$. We will find the description of this action in terms of Newton polynomials convenient for our purposes, so we review how this works. Our choice to describe the action in this way was inspired by \cite[Section 5]{Baker}.

We define the Newton polynomials $N_n (t) = N_n (t_1, \dots, t_n) \in \ZZ [t_1, \dots, t_n]$ by setting $N_1 (t) = t_1$ and inductively letting
$$N_n (t) = t_1 N_{n-1} (t) - t_2 N_{n-2} (t) + \dots + (-1)^{n-2} t_{n-1} N_1 (t) + (-1)^{n-1} n t_n. $$

Then the following useful relation holds:
$$N_{pn} (t) = (N_n (t))^p \mod p.$$

We let $N_n (b) \in \mathrm{H}_* \mathrm{MU}$ be defined by setting $t_n = b_n$, and let $N_n (\xi) \in \mathrm{H}_* \mathrm{MU}$ be defined by setting $t_{p^k -1} = \xi_k$ and the other $t_n$ to zero.
When viewed as an element of $\H_* \mathrm{BU}$ under the Thom isomorphism, $N_n (b)$ generates the subgroup of primitive elements in degree $2n$.
Writing out the recurrence for $N_{p^k-1} (\xi)$ shows that $N_{p^k-1} (\xi) = - \xibar_k$ where $x \mapsto \overline{x}$ is the conjugation in the Hopf algebra $\mathrm{H}_* \mathrm{H}$.

Kochman \cite[Theorem 5]{Koch} showed that the action of the Dyer-Lashof operations on $N_n(b)$ is described by the formula:
$$Q^r N_n (b) = (-1)^{r+n} {r-1 \choose n-1} N_{n+r(p-1)} (b).$$
Since the orientation $\mathrm{MU} \to \mathrm{H}$ maps $b_{p^k-1}$ to $\xi_k$ and the other $b_n$ to zero, it maps $N_n (b)$ to $N_n (\xi)$ and so we also have:
$$Q^r N_n (\xi) = (-1)^{r+n} {r-1 \choose n-1} N_{n+r(p-1)} (\xi).$$ 
Using $N_{p^k -1} (\xi) = -\xibar_k$, we get:
$$Q^r \xibar_k = (-1)^{r+1} {r-1 \choose p^k - 2} N_{p^k - 1 + r(p-1)} (\xi).$$

These formulas provide good control over the action of the Dyer-Lashof operations on $\H_* \H$ and primitive elements in $\H_* \mathrm{BU}$.
We will also need a good understanding of the action of the Dyer-Lashof operations on the indecomposable generators of $\H_* \MU \cong \H_* \mathrm{BU}$.
It is immediately from the definitions that $N_n (b)$ is such a generator when $n$ is relatively prime to $p$, but is decomposable when $n$ is divisible by $p$.
In degrees divisible by $p$, the work of Lance determines generators on which the action of the Dyer-Lashof operations is convenient \cite{Lance}.
In the integral homology $\H_* (\mathrm{BU}, \ZZ_{(p)})$, Lance inductively defines classes $a_{n,k}$ for $k \geq 0$ and $n$ coprime to $p$ by the formula
\[N_{np^k} (b) = a_{n,0}^{p^k} + p a_{n,1}^{p^{k-1}} + \dots + p^k a_{n,k},\]
so that $\abs{a_{n,k}} = 2p^k n$.
In the above, we use $N_n (b)$ to refer to the integral Newton polynomial, which generates the subgroup of primitive elements in $\H_{2n} (\mathrm{BU}; \ZZ)$.
When $k = 0$, we simply have $a_{n,0} = N_n (b)$.
Lance proves that these classes lie in $\H_* (\mathrm{BU}; \ZZ_{(p)})$ and are indecomposable generators \cite[Theorem 2.3]{Lance}.

To determine the action of the Dyer-Lashof algebra on the mod $p$ reductions $\overline{a}_{n,k} \in \H_{2p^k n} (\mathrm{BU})$, Lance proves in \cite[Theorem 4.2]{Lance} that there is a $p$-local integral lift of the Dyer-Lashof operations on $\H_* (\mathrm{BU})$ to operations $Q^i : \H_{n} (\mathrm{BU}; \ZZ_{(p)}) \to \H_{n+2(p-1)i} (\mathrm{BU}; \ZZ_{(p)})$ which satisfy the following properties:
\begin{enumerate}
  \item $Q^i$ is a linear map.
  \item The Cartan formula $Q^i (xy) = \sum_{j+k = i} Q^j(x) Q^k (y)$ is satisfied.
  \item We have $Q^r N_n (b) = (-1)^{r+n} {r-1 \choose n-1} N_{n+r(p-1)} (b)$.
\end{enumerate}
Using these properties, one may inductively determine the action of the integral lifts of the Dyer-Lashof operations on $a_{n,k}$.

Using the above formulas and the Cartan formula, we may deduce the following two propositions by direct calculation. These propositions summarize what we need to know about the action of the Dyer-Lashof operations on $\H_* \H$ and $\H_* \MU$, respectively.

\begin{prop}\label{StDL}
	In the dual Steenrod algebra $\mathrm{H}_* \mathrm{H}$, the following identities hold:
	\begin{align*}
		&Q^{p^2} \xibar_1 = (\xibar_1 ^{p-1})^p Q^{p} \xibar_1 \\
		&Q^{p^2+i} \xibar_1 = 0 \,\, \mathrm{ for } \, \,i = 1, \dots, p-2 \\
		&Q^{p^2+p-1} \xibar_1 = -(Q^p (\xibar_1))^p \\
		&Q^{p^2-p+1} (\xibar_1 ^{p-1}) = (\xibar_1)^{p^2} \\
		&Q^{p^2 +pi} Q^p \xibar_1 = 0 \,\, \mathrm{ for } \, \, i = 1, \dots, p-1 \\
		&Q^{2p} \xibar_1 = -\xibar_1 ^p Q^p (\xibar_1).
	\end{align*}
\end{prop}

\begin{prop}\label{MUDL}
	The following identities hold in $\mathrm{H}_* (\mathrm{MU})$:
  \begin{align*}
    Q^{p^2-p+1} (\overline{a}_{p-1,1}) = & \, \,Q^{p^2} N_{p-1} (b) - N_{p-1} (b) ^{p(p-1)} Q^p (N_{p-1}(b)) \\
    Q^{p^2-p+1} (N_{p-1} (b)^p) = & \, \, 0 \\
    Q^{p^2+i} N_{p-1} (b) = & \, \, 0 \,\, \mathrm{ for } \, \,i = 1, \dots, p-2 \\
    Q^{p^2+p-1} N_{p-1} (b) = & \, \, -(Q^p (N_{p-1} (b)))^p \\
    Q^{p^2-p+1} (N_{p-1} (b)^{p-1}) = & \, \, (N_{p-1} (b))^{p(p-2)} (N_{2(p-1)} (b))^p \\
    Q^{p^2 +pi} Q^p N_{p-1} (b) = & \,\, 0 \,\, \mathrm{ for } \, \, i = 1, \dots, p-1 \\
    Q^{2p} N_{p-1} (b) = & \, \, - N_{(2p+1) (p-1)} (b) \\
    Q^p N_{p-1} (b) = & \, \, N_{(p+1)(p-1)} \\
    Q^{p^2+1} (\overline{a}_{p-1,1}) =& \,\, N_{p-1} (b)^{p(p-1)} N_{(2p+1)(p-1)} (b) \\
    & - (N_{p-1}(b)^{p-2} N_{2(p-1)} (b))^p Q^p N_{p-1}(b) \\
    Q^{p^2+1} (N_{p-1} (b)^p) = & \, \, 0.
  \end{align*}
%
%
\end{prop}

Finally, we identify some elements in the kernel of $\H_* \MU \to \H_* \H$.

\begin{prop} \label{prop:MUHker}
  The following classes lie in the kernel of $\H_* \MU \to \H_* \H$:
  \begin{enumerate}
    \item $N_{2(p-1)} (b) + N_{p-1} (b)^2$ 
    \item $- \overline{a}_{p-1,1} - N_{p-1} (b)^p$
    \item $N_{(2p+1)(p-1)} (b) + N_{(p+1)(p-1)} (b) N_{p-1} (b)^p.$
  \end{enumerate}
\end{prop}

\begin{proof}
  Cases (1) and (3) are elementary computations using the definition of the Newton polynomials and the fact that $b_{p^i -1} \mapsto \xi_i$ and the rest of the $b_i$ go to zero.
  To prove case (2), we recall that by definition we have
  \[p a_{p-1,1} = N_{p(p-1)} (b) - N_{p-1} (b)^p.\]
  Since $b_1, \dots, b_{p-2}, b_p , \dots, b_{p(p-1)}$ go to zero under $\H_* \MU \to \H_* \H$, it suffices to compute $a_{p-1,1}$ modulo $I = (b_1, \dots, b_{p-2}, b_p, \dots b_{p(p-1)})$.
  It is easy to see that $N_{(p-1)} (b) \equiv -(p-1) b_{p-1} \mod I$ and $N_{i(p-1)} (b) \equiv - b_{p-1} N_{(i-1)(p-1)} (b) \mod I$ for $i \leq p$, so that $N_{p(p-1)} (b) \equiv -(p-1) b_{p-1} ^p \mod I$.
  We therefore have
  \begin{align*}
    p a_{p-1,1} &\equiv - (p-1) b_{p-1}^p - (-(p-1)b_{p-1})^p \equiv ((p-1)^p - (p-1))b_{p-1} ^p  \\
    &\equiv -p b_{p-1} ^p \mod (p^2, I),
  \end{align*}
  so that
  \[\overline{a}_{p-1,1} \equiv - b_{p-1}^p \mod (p, I).\]
  Since $N_{p-1} (b) \equiv b_{p-1} \mod (p,b_1, \dots, b_{p-2})$, we are done.
\end{proof}

%
%
%

\subsection{A relation among power operations}\label{SecRel}

We will define the secondary operation of interest to us in terms of the following relation between primary power operations.

\begin{prop}\label{relation}

	Let $R$ be an $\mathbb{E}_{2(p^2+2)}$-$\mathrm{H}$-algebra and $x \in \pi_{2(p-1)} (R)$. Define classes $a_i$, $i=0, \dots, p-1$; $b$; $c_i$, $i=1, \dots, p$ in $\pi_* (R)$ by the following formulae:
	\begin{align*}
		&a_0 = Q^{p^2}x - (x^{p-1})^p Q^p x \\
		&a_i = Q^{p^2 + i} x \,\, \mathrm{ for } \,\, i = 1, \dots, p-2 \\
		&a_{p-1} = Q^{p^2+p-1} x + (Q^p x)^p \\
		&b = Q^{p^2-p+1} (x^{p-1}) - x^{p^2} \\
		&c_i = Q^{p^2 + pi} Q^p x \,\, \mathrm{ for } \,\, i = 1, \dots, p-1 \\
		&c_p = Q^{2p} x + (Q^p x) x^p
	\end{align*}
	Then the following identity holds:
	\begin{align*}
		0 = & Q^{p^3+p} a_0 + \sum_{i=1}^{p-2} (-1)^i Q^{p^3+p-i} a_i + Q^{p^3+1} a_{p-1} +\\
		& b^p Q^{p^2} Q^p x + \sum_{i=1} ^{p-1} (Q^{p^2 - p - i + 1} (x^{p-1}))^p c_i + (x^{p-1})^{p^2} Q^{2p^2-p} c_p
	\end{align*}
\end{prop}

\begin{proof}
  It is elementary to check that the operation in this relation which takes the greatest $n$ to be defined on $\mathbb{E}_n$-$\mathrm{H}$-algebras is $Q^{p^3+p} a_0$. Since $\abs{a_0} = 2(p-1)(p^2+1)$, we conclude from \Cref{DLWhen} that this is defined and satisfies the usual properties whenever
	\begin{align*}
		n \geq 2(p^3+p) - 2(p-1)(p^2+1) + 2 = 2(p^2+2),
	\end{align*}
	so that this relation is defined for $\mathbb{E}_{2(p^2+2)}$-$\mathrm{H}$-algebras.

	The desired identity reduces to the following identities, which may be deduced from the Adem relations, the instability relations, and the Cartan formula:
	\begin{align*}
		&Q^{p^3+p} Q^{p^2} x = \sum_{i=1}^{p-1} (-1)^{i+1} Q^{p^3+p-i} Q^{p^2+i} x \\
		&Q^{p^3+1} (( Q^p x)^p) = 0 \\
		&Q^{p^3+p} ((x^{p-1})^p Q^p x) = \sum_{i=0}^p (Q^{p^2-p-i+1} (x^{p-1}))^p Q^{p^2+pi} Q^p x \\
		&Q^{2p^2} Q^p x = Q^{2p^2 - p} Q^{2p} x \\
		&Q^{2p^2 - p} (x^p Q^p x) = x^{p^2} Q^{p^2} Q^p x. \qedhere
	\end{align*}
\end{proof}

Let the symbols $a_i$, $i=0, \dots, p-1$; $b$; $c_j$, $j=1, \dots, p$ have the gradings
of the elements in Proposition \hyperref[relation]{\ref*{relation}}, and let $d$ have the grading of the relation described there.
Then the relation above determines maps $$Q : \mathbb{P}_\mathrm{H} ^{2(p^2+2)} (x, a_0, \dots, a_{p-1}, b, c_0, \dots, c_{p-1}) \to \mathbb{P}_\mathrm{H} ^{2(p^2+2)} (x)$$ and $$R : \mathbb{P}_\mathrm{H} ^{2(p^2+2)} (x,d) \to \mathbb{P}_\mathrm{H} ^{2(p^2+2)} (x, a_0, \dots, a_{p-1}, b, c_0, \dots, c_{p-1})$$ in $\mathcal{C}$ (in particular, $R$ is a map of augmented objects) such that the composition $Q \circ R$ is nullhomotopic in $\mathcal{C}$.
	

\begin{defin}
  From this point on, we regard $\H \wedge \H$ as an object of $\mathcal{C}$ by first lifting it to $\mathcal{D}$ by setting $x = \xibar_1$ and then regarding it as a neither pointed nor augmented object.
\end{defin}

	\begin{prop}\label{bracketdefined}

		The bracket $\langle \xibar_1, Q, R \rangle$ is defined in $\mathrm{H}_* \mathrm{H}$ and has zero indeterminacy in the quotient $\mathcal{E}_* = \Lambda_{\FF_p} (\tau_0, \tau_1, \dots)$ of $\mathrm{H}_* \mathrm{H}$.

	\end{prop}

	\begin{proof}

		To show that the bracket is defined, we need to show that $Q(\xibar_1) = 0$. This is equivalent to \Cref{StDL}.

    The indeterminacy comes from degree $(2p^3+2p^2+2p+1)$ $\mathbb{E}_{2(p^2+2)}$-$\mathrm{H}$-algebra homotopy operations applied to $\xibar_1$ and from the image of the suspended operation $\sigma R$. Since $\xibar_1$ lies in the image of the map $\H_* \MU \to \H_* \H$ induced by the $\mathbb{E}_\infty$-map $\MU \to \H$, the indeterminacy of the first type also lies in this image and hence maps to zero in $\mathcal{E}_*$.
%

    By the odd-primary analogue of \cite[Proposition 2.6.13]{BPtwo}, $\sigma R$ is equal to
    \begin{align*}
      Q^{p^3 +p} \sigma a_0 &+ \sum_{i =1} ^{p-2} (-1)^i Q^{p^3+p-i} \sigma a_i + Q^{p^3+1} \sigma a_{p-1}\\
      &+ \sum_{i=1} ^{p-1} (Q^{p^2-p-i+1}(x^{p-1}))^p \sigma c_i + (x^{p-1})^{p^2} Q^{2p^2-p} \sigma c_p,
    \end{align*}
    where the $\sigma a_i$ and $\sigma c_i$ are variables in degree one higher than $a_i$ and $c_i$, respectively. In particular, $\abs{\sigma a_i} = 2(p^2+i+1)(p-1) + 1$ for $i = 0, \dots, p-1$. Since $\H_* \H$ is decomposable in these degrees, we conclude that the second sort of indeterminacy must be decomposable in $\H_* \H$. Since there are no nonzero decomposables in $\mathcal{E}_*$ in degree $2p^4-1$, we conclude that the indeterminacy must actually be trivial in $\mathcal{E}_*$.
\end{proof}

\subsection{Computation of the secondary operation}\label{compsecond}

  To compute this operation, we will first juggle it into a functional operation for the map $\mathrm{H} \wedge \mathrm{MU} \to \mathrm{H} \wedge \mathrm{H}$. To this end, we define maps in $\CC$ (in this case, maps augmented over $\mathbb{P}_\mathrm{H} ^{2(p^2+2)} (x)$):
	\begin{align*}
    &\mu : \mathbb{P}_\mathrm{H} ^{2(p^2+2)} (x, a_0, \dots, a_{p-1}, b, c_0, \dots, c_{p-1}) \to \mathbb{P}_\mathrm{H} ^{2(p^2+2)} (x,y_{2(p-1)},y_{p(p-1)},y_{(2p+1)(p-1)}) \\
    &\overline{Q} : \mathbb{P}_\mathrm{H} ^{2(p^2+2)} (x, z_{p^3-1}) \to \mathbb{P}_\mathrm{H} ^{2(p^2+2)} (x,y_{2(p-1)},y_{p(p-1)},y_{(2p+1)(p-1)}) \\
		&\nu : \mathbb{P}_\mathrm{H} ^{2(p^2+2)} (x,d) \to \mathbb{P}_\mathrm{H} ^{2(p^2+2)} (x, z_{p^3-1}),
	\end{align*}
  where $\abs{y_i} = 2i$ and $\abs{z_{p^3-1}} = 2(p^3-1)$, by:
	\begin{align*}
		\mu(a_0) &= Q^{p^2-p+1} y_{p(p-1)} \\
		\mu(a_i) &= 0 \,\, \mathrm{ for } \,\, i \neq 0 \\
    \mu(b) &= -x^{p(p-2)} y_{2(p-1)}^p \\
    \mu(c_i) &= 0 \text{ for } i \neq p \\
    \mu(c_p) &= y_{(2p+1)(p-1)} \\
    \overline{Q} (z_{p^3-1}) &= Q^{p^2+1} y_{p(p-1)} - (x^{p-2} y_{2(p-1)})^p Q^p x + x^{p(p-1)} y_{(2p+1)(p-1)} \\
    \nu(d) &= Q^{p^3} z_{p^3-1}.
	\end{align*}
%
%
%
  \begin{defin}
    From this point on, we regard $\H \wedge \MU$ as an object of $\mathcal{C}$ by first lifting it to $\mathcal{D}$ by setting $x = - N_{p-1} (b)$ and then regarding it as a neither pointed nor augmented object.
  \end{defin}

  \begin{prop} \label{prop:good}
		There is an identity $\mu R = \overline{Q} \nu$ and a homotopy commutative diagram
		\begin{center}
		\begin{tikzcd}[row sep = huge, column sep = large]
			\mathbb{P}_\mathrm{H} ^{2(p^2+2)} (x, a_0, \dots, a_{p-1}, b, c_0, \dots, c_{p-1}) \arrow{r}{Q} \arrow{d}{\mu} & \mathbb{P}_\mathrm{H} ^{2(p^2+2)} (x) \arrow{d}{-N_{p-1}(b)} \arrow{rd}{\xibar_1} & \\
      \mathbb{P}_\mathrm{H} ^{2(p^2+2)} (x, y_{2(p-1)}, y_{p(p-1)}, y_{(2p+1)(p-1)}) \arrow{r}{f} & \mathrm{H} \wedge \mathrm{MU} \arrow{r}{p} & \mathrm{H} \wedge \mathrm{H},
		\end{tikzcd}
		\end{center}
    where $f$ is the map defined by the following:
    \begin{align*}
      x &\mapsto -N_{p-1} (b) \\
      y_{2(p-1)} &\mapsto N_{2(p-1)} (b) + N_{p-1}(b)^2 \\
      y_{p(p-1)} &\mapsto -\overline{a}_{p-1,1} - N_{p-1} (b)^p \\
      y_{(2p+1)(p-1)} &\mapsto N_{(2p+1)(p-1)} (b) + N_{(p+1)(p-1)}(b) N_{p-1} (b)^p.
    \end{align*}
    Furthermore, the composites $p \circ f$ and $f \circ \overline{Q}$ are null in the category $\mathcal{C}$.
	\end{prop}

	\begin{proof}

    The right triangle of the diagram commutes because $\xibar_1 = - N_{p-1} (\xi)$ and hence $p(-N_{p-1} (b)) = \xibar_1$.
    The left square commutes by the first eight equations of \Cref{MUDL}.

    The identity $\mu R = \overline{Q} \nu$ follows from the relations
    \[Q^{p^3+p} Q^{p^2-p+1} y_{p(p-1)} = Q^{p^3} Q^{p^2+1} y_{p(p-1)}\]
    \[Q^{p^3} ((x^{p-2} y_{2(p-1)})^p Q^p x) = (x^{p(p-2)} y_{2(p-1)}^p)^p Q^{p^2} Q^p x\]
    \[Q^{p^3} (x^{p(p-1)} y_{(2p+1)(p-1)}) = x^{p^2(p-1)} Q^{2p^2-p} y_{(2p+1)(p-1)},\]
    which are readily deduced from the Adem, Cartan and instability relations.

    Finally, the fact that $p \circ f$ and $f \circ \overline{Q}$ are null in the category $\mathcal{C}$ follows from \Cref{prop:MUHker} and the final three equations of \Cref{MUDL}, respectively.
%
%
	%
%
	\end{proof}

	\begin{prop}\label{CompProp1}
		There is an equality $\<\xibar_1, Q, R\> \equiv Q^{p^3} (\<p,f,\overline{Q}\>)$ in $\mathcal{E}_*$.
	\end{prop}

	\begin{proof}
    Applying \Cref{prop:good} repeatedly, we see that the juggling relations for brackets imply the following sequence of identities:
		\begin{align*}
      \< \xibar_1, Q, R \> &= \<p (-N_{p-1} (b)), Q, R\> \\
      &\subset \<p, (-N_{p-1} (b)) Q, R\> \\
			&= \<p, f \mu, R \> \\
			&\supset \<p, f, \mu R\> \\
			&= \<p,f, \overline{Q} \nu\> \\
			&\supset \<p,f,\overline{Q} \> \sigma\nu.
		\end{align*}

    Let us show that the image in $\mathcal{E}_*$ of each bracket appearing above is equal to a single element. 
    To do this, it suffices to show that this is true for $\< p, (-N_{p-1} (b)) Q, R \>$.
    This follows from the argument of \Cref{bracketdefined}.
    To conclude, we note that $\sigma \nu(d) = Q^{p^3} \sigma z_{p^3-1}$.
%
%
%
%
	\end{proof}

	Finally, we compute the bracket $\<p,f,\overline{Q}\>$ by means of Theorem \hyperref[DLMUSt]{\ref*{DLMUSt}}.

  \begin{defin}
    Since the map $i : \H \wedge \H \to \H \wedge_{\MU} \H$ sends $x = \xibar_1$ to zero, the composite map $\mathbb{P}_\mathrm{H} ^{2(p^2+2)} (x) \xrightarrow{\xibar_1} \H \wedge \H \xrightarrow{i} \H \wedge_{\MU} \H$ factors through the augmentation $\mathbb{P}_\mathrm{H} ^{2(p^2+2)} (x) \to \H$. As a consequence, we may regard $\H \wedge_{\MU} \H$ as a pointed object of $\mathcal{C}$ and the map $i$ as a morphism in $\mathcal{C}$.
  \end{defin}

	\begin{prop}\label{CompProp2}
		There is an inclusion $C \tau_3 \in \<p,f,\overline{Q} \>$ after reducing to $\mathcal{E}_*$ for some $C \in \FF_p ^\times$.
	\end{prop}

	\begin{proof}
		By noting that each pair of maps in the diagram
		\begin{align*}
			\mathbb{P}_\mathrm{H} ^{2(p^2+2)} (x, z_{p^3-1}) \xrightarrow{\overline{Q}} \mathbb{P}_\mathrm{H} ^{2(p^2+2)} (x,y_{p(p-1)}) \xrightarrow{f} \mathrm{H} \wedge \mathrm{MU} \xrightarrow{p} \mathrm{H} \wedge \mathrm{H} \xrightarrow{i} \mathrm{H} \wedge_{\mathrm{MU}} \mathrm{H}
		\end{align*}
    compose to a nullhomotopic map in $\CC$ (the first two by \Cref{prop:good} and the third by definition), we find that we are allowed to apply the Peterson-Stein formula to obtain the equality $$i \<p,f,\overline{Q}\> = -\<i,p,f\> \sigma \overline{Q}.$$
    Since both $\overline{a}_{p-1,1}$ and the Hurewicz image of $[\CP^{p(p-1)}]$ are indecomposable generators in degree $2p(p-1)$, we must have
    \[\overline{a}_{p-1,1} \equiv C [\CP^{p(p-1)}] \mod \mathrm{decomposables}\]
    for some $C \in \FF_p ^\times$.
    It follows from \cite[Proposition 2.7.5]{BPtwo} that
    \[-C\sigma[\CP^{p(p-1)}] = \sigma (-\overline{a}_{p-1,1}) = \sigma(-\overline{a}_{p-1,1} - N_{p-1} (b)^p) \in \<i,p,f\>.\]
    Since $\sigma \overline{Q} = Q^{p^2+1} \sigma y_{p(p-1)} + x^{p(p-1)} \sigma y_{(2p+1)(p-1)}$ by the odd-primary version of \cite[Proposition 2.6.13]{BPtwo}, it follows from \Cref{DLMUSt} and the fact that $x=0 \in \pi_* \H \otimes_{\MU} \H$ that
    \[-C\sigma v_3 \in - \< i, p ,f \> \sigma \overline{Q} = i \< p,f,\overline{Q} \>.\]
    Since $\pi_* (i)$ factors as
    \[\H_* \H \twoheadrightarrow \mathcal{E}_* \hookrightarrow \pi_* (\H \wedge_{\MU} \H)\]
    and $\tau_3$ spans $\mathcal{E}_{2p^3-1}$, we conclude that after reducing to $\mathcal{E}_*$ we must have
    \[C \tau_3 \in \<p,f,\overline{Q}\>\]
    for some, possibly different, $C \in \FF_p ^\times$.
%
%
%
%
	\end{proof}

	\begin{cor}
		For some $C \in \FF_p ^\times$, we have $\<\xibar_1,Q,R\> \equiv C \tau_4$  in $\mathcal{E}_*$.
	\end{cor}

	\begin{proof}
    Combine Propositions \hyperref[CompProp1]{\ref*{CompProp1}} and \hyperref[CompProp2]{\ref*{CompProp2}} with the operation $Q^{p^3} \tau_{3} \equiv \tau_4$ in $\mathcal{E}_*$, which may be deduced from \cite[Theorem III.2.3]{Hinf}.
	\end{proof}

    Since maps of $\mathbb{E}_{2(p^2+2)}$-ring spectra must preserve secondary power operations \cite[Proposition 2.6.7]{BPtwo}, we obtain the following corollary.

	\begin{cor}\label{FinCor}
		Let $R$ be an $\mathbb{E}_{2(p^2+2)}$-ring spectrum and let $R \to \mathrm{H}$ be a map of $\mathbb{E}_{2(p^2+2)}$-ring spectra. Then if the induced map on homology $\mathrm{H}_* R \to \mathrm{H}_* \mathrm{H}$ is injective in degrees less than or equal to $2(2p^2+1)(p-1)$ and contains $\xibar_1$ in its image, then $\tau_4$ must also be in the image of the composite $\mathrm{H}_* R \to \mathrm{H}_* \mathrm{H} \to \mathcal{E}_*$.
	\end{cor}

  \begin{rmk}
    In \Cref{FinCor}, the injectivity assumption ensures that for a lift $x \in \H_* R$ of $\xibar_1$, we must have $Q(x) = 0$. The quantity $2(2p^2+1)(p-1)$ arises as the degree of the highest term of the operation $Q$, which is given by $Q(c_{p-1}) = Q^{p^2+(p-1)p} Q^p x$.
  \end{rmk}

	We conclude by deducing Theorem \hyperref[MainThm]{\ref*{MainThm}} from Corollary \hyperref[FinCor]{\ref*{FinCor}}.

	\begin{proof}[Proof of \Cref{MainThm}]
		Assume that $\mathrm{BP}$ were an $\mathbb{E}_{2(p^2+2)}$-ring spectrum. Since the Postnikov tower of an $\mathbb{E}_n$-ring spectrum naturally lifts to a tower of $\mathbb{E}_n$-ring spectra, there is a map of $\mathbb{E}_{2(p^2+2)}$-ring spectra $$\mathrm{BP} \to \tau_{\leq 0} \mathrm{BP} \cong \mathrm{H}\ZZ_{(p)} \to \mathrm{H}$$ which induces the inclusion $$\FF_p [\xi_1, \xi_2, \dots] \hookrightarrow \Lambda_{\FF_p} (\tau_0, \tau_1, \dots) \otimes \FF_p [\xi_1, \xi_2, \dots]$$ upon taking homology. In particular, the map is injective and contains $\xibar_1$ in its image. However, $\tau_4$ cannot be in the image of $\mathrm{H}_* \mathrm{BP} \to \mathrm{H}_* \mathrm{H} \to \mathcal{E}_*$ because this composite is zero.

		The case of $\mathrm{BP}\<n\>$ for $n \geq 4$ is analogous, using the fact that $$\mathrm{H}_* (\mathrm{BP}\<n\>) \cong \Lambda_{\FF_p} (\tau_{n+1}, \tau_{n+2}, \dots) \otimes \FF_p [\xi_1, \xi_2, \dots].$$
		Finally, taking $p$-completions makes no difference because we are only working with mod $p$ homology in the first place.
	\end{proof}

\bibliographystyle{alpha}
\bibliography{main}

\begin{thebibliography}{BMMS86}

\bibitem[AL17]{addBPn}
Vigleik Angeltveit and John~A. Lind.
\newblock Uniqueness of {$BP \langle n\rangle$}.
\newblock {\em J. Homotopy Relat. Struct.}, 12(1):17--30, 2017.

\bibitem[Bak15]{Baker}
Andrew Baker.
\newblock Power operations and coactions in highly commutative homology
  theories.
\newblock {\em Publ. Res. Inst. Math. Sci.}, 51(2):237--272, 2015.

\bibitem[BM13]{E4}
Maria Basterra and Michael~A. Mandell.
\newblock The multiplication on {BP}.
\newblock {\em J. Topol.}, 6(2):285--310, 2013.

\bibitem[BMMS86]{Hinf}
R.~R. Bruner, J.~P. May, J.~E. McClure, and M.~Steinberger.
\newblock {\em {$H_\infty $} ring spectra and their applications}, volume 1176
  of {\em Lecture Notes in Mathematics}.
\newblock Springer-Verlag, Berlin, 1986.

\bibitem[CLM76]{CLM}
Frederick~R. Cohen, Thomas~J. Lada, and J.~Peter May.
\newblock {\em The homology of iterated loop spaces}.
\newblock Lecture Notes in Mathematics, Vol. 533. Springer-Verlag, Berlin-New
  York, 1976.

\bibitem[HL10]{HL}
Michael Hill and Tyler Lawson.
\newblock Automorphic forms and cohomology theories on {S}himura curves of
  small discriminant.
\newblock {\em Adv. Math.}, 225(2):1013--1045, 2010.

\bibitem[JN10]{JN}
Niles Johnson and Justin Noel.
\newblock For complex orientations preserving power operations,
  {$p$}-typicality is atypical.
\newblock {\em Topology Appl.}, 157(14):2271--2288, 2010.

\bibitem[Koc73]{Koch}
Stanley~O. Kochman.
\newblock Homology of the classical groups over the {D}yer-{L}ashof algebra.
\newblock {\em Trans. Amer. Math. Soc.}, 185:83--136, 1973.

\bibitem[{Kri}95]{Kriz}
I.~{Kriz}.
\newblock {Towers of $E_\infty$ ring spectra with an application to $BP$}.
\newblock {\em Unpublished Preprint}, 1995.

\bibitem[Lah20]{HomOpRev}
Anssi Lahtinen.
\newblock Homology operations revisited.
\newblock 2020.
\newblock \href{https://arxiv.org/abs/2001.02781}{arXiv:2001.02781}.

\bibitem[Lan83]{Lance}
Timothy Lance.
\newblock Steenrod and {D}yer-{L}ashof operations on {$B{\bf U}$}.
\newblock {\em Trans. Amer. Math. Soc.}, 276(2):497--510, 1983.

\bibitem[Law18]{BPtwo}
Tyler Lawson.
\newblock Secondary power operations and the {B}rown-{P}eterson spectrum at the
  prime 2.
\newblock {\em Ann. of Math. (2)}, 188(2):513--576, 2018.

\bibitem[LN12]{LN}
Tyler Lawson and Niko Naumann.
\newblock Commutativity conditions for truncated {B}rown-{P}eterson spectra of
  height 2.
\newblock {\em J. Topol.}, 5(1):137--168, 2012.

\bibitem[Lur]{HA}
Jacob Lurie.
\newblock {\em {Higher algebra}}.
\newblock \newline {A}vailable at:
  \href{https://www.math.ias.edu/~lurie/papers/HA.pdf}{https://www.math.ias.edu/~lurie/papers/HA.pdf}.
  Theorem numbers referenced follow the Sept. 18, 2017 edition.

\bibitem[Lur09]{HTT}
Jacob Lurie.
\newblock {\em Higher topos theory}, volume 170 of {\em Annals of Mathematics
  Studies}.
\newblock Princeton University Press, Princeton, NJ, 2009.

\bibitem[May70]{MaySt}
J.~Peter May.
\newblock A general algebraic approach to {S}teenrod operations.
\newblock In {\em The {S}teenrod {A}lgebra and its {A}pplications ({P}roc.
  {C}onf. to {C}elebrate {N}. {E}. {S}teenrod's {S}ixtieth {B}irthday,
  {B}attelle {M}emorial {I}nst., {C}olumbus, {O}hio, 1970)}, Lecture Notes in
  Mathematics, Vol. 168, pages 153--231. Springer, Berlin, 1970.

\bibitem[May75]{May}
J.~P. May.
\newblock Problems in infinite loop space theory.
\newblock In {\em Conference on homotopy theory ({E}vanston, {I}ll., 1974)},
  volume~1 of {\em Notas Mat. Simpos.}, pages 111--125. Soc. Mat. Mexicana,
  M\'exico, 1975.

\bibitem[PS18]{rect}
Dmitri Pavlov and Jakob Scholbach.
\newblock Admissibility and rectification of colored symmetric operads.
\newblock {\em J. Topol.}, 11(3):559--601, 2018.

\bibitem[Rav86]{GreenBook}
Douglas~C. Ravenel.
\newblock {\em Complex cobordism and stable homotopy groups of spheres}, volume
  121 of {\em Pure and Applied Mathematics}.
\newblock Academic Press, Inc., Orlando, FL, 1986.

\end{thebibliography}

\end{document}